\documentclass[a4paper,11pt]{article}
\usepackage{amsfonts}
\usepackage{amssymb}
\usepackage{amsmath}
\usepackage{mathtools}
\usepackage{mathrsfs}
\usepackage{bbm}
\usepackage{xcolor}
\usepackage[utf8]{inputenc}
\usepackage[margin=1.5in]{geometry}
\usepackage[T1]{fontenc}
\usepackage{amsthm}
\usepackage[english]{babel}
\usepackage{graphicx}
\usepackage{tikz-cd}
\usepackage{tikz}
\usepackage{dsfont}
\usetikzlibrary{matrix, arrows}

\newtheorem{theorem}{Theorem}[subsection]
\newtheorem{prop}[theorem]{Proposition}
\newtheorem{lemma}[theorem]{Lemma}
\newtheorem{corollary}[theorem]{Corollary}
\theoremstyle{definition}

\theoremstyle{definition}
\newtheorem{remark}[theorem]{Remark}
\theoremstyle{definition}
\newtheorem{defi}[theorem]{Definition}

\newcommand\bsni{\bigskip\noindent}

\newcommand\bba{\mathbb{A}}

\newcommand\bbc{\mathbb{C}}

\newcommand\bbn{\mathbb{N}}
\newcommand\bbp{\mathbb{P}}
\newcommand\bbq{\mathbb{Q}}
\newcommand\bbr{\mathbb{R}}
\newcommand\bbs{\mathbb{S}}

\newcommand\cE{\mathcal{E}}

\newcommand\cH{\mathcal{H}}

\newcommand\cN{\mathcal{N}}
\newcommand\cO{\mathcal{O}}

\newcommand\mi{^{-1}}

\DeclareMathOperator{\vol}{vol}

\newcommand\field{\mathfrak{K}}
\newcommand\fieldext{\mathfrak{L}}
\newcommand\norm{\zeta}

\newcommand\PSH{\mathrm{PSH}}
\newcommand\MA{\,\mathrm{MA}}
\newcommand\grnorm{\norm_\bullet}
\newcommand\grnormspace{\mathcal{N}_\bullet}
\newcommand\normspace{\mathcal{N}}

\newcommand\basis{\mathfrak{s}}
\newcommand\basislong{\mathfrak{s}=(s_i)_i}
\newcommand\refmetric{\phi_{\mathrm{ref}}}
\newcommand\an{^{\mathrm{an}}}
\newcommand\FS{\mathrm{FS}}
\newcommand\N{\mathrm{N}}

\newcommand\path{\mathfrak{P}}

\author{REBOULET Rémi}
\title{Plurisubharmonic geodesics in spaces of non-Archimedean metrics of finite energy}
\begin{document}
\maketitle

\begin{abstract}\noindent Given a polarized projective variety $(X,L)$ over any non-Archimedean field, assuming continuity of envelopes, we show that the space of finite-energy metrics on $L$ is a geodesic metric space, where geodesics are given as maximal psh segments. Given two continuous psh metrics, we show that the maximal segment joining them is furthermore continuous.
\end{abstract}

\tableofcontents

\section*{Introduction.}

\paragraph{The complex picture.} Consider a smooth, projective complex manifold $X$ endowed with an ample line bundle $L$. We denote by $dd^c\phi$ the curvature current of a smooth metric $\phi$ on $L$. If this current is (strictly) positive, $\phi$ is said to be (strictly) plurisubharmonic, or \textit{psh} for short, and we define $\cH(L)$ to be the space of strictly plurisubharmonic smooth metrics on $L$. The study of metric structures on $\cH(L)$ has a long history (\cite{mabuchi}, \cite{semmes}, \cite{xxchen}...), and one is interested in extending results from this case to singular metrics.

\bsni A way to define a natural class of singular metrics to study goes as follows. Given two metrics in $\cH(L)$, let their relative Monge-Ampère energy be
$$E(\phi_0,\phi_1)=\frac{1}{d+1}\sum_{i=0}^d \int_X (\phi_0-\phi_1)\MA(\phi_0^{(i)},\phi_1^{(d-i)}),$$
where $d=\dim X$, and 
$$\MA(\phi_0^{(i)},\phi_1^{(d-i)})=\vol(L)\mi (dd^c\phi_0)^i\wedge(dd^c\phi_1)^{d-i}$$
are the mixed Monge-Ampère measures of $\phi_0$ and $\phi_1$. Fixing a reference metric $\refmetric\in\cH(L)$, we set
$$E(\phi)=E(\phi,\refmetric)$$
for $\phi\in\cH(L)$. By \cite[Proposition 2.4]{begz}, the functional $E$ admits a nice extension to the space of psh metrics on $L$, and we define the set $\cE^1(L)$ of finite-energy metrics on $L$ to be the set of psh metrics $\phi$ with $E(\phi)>-\infty$.

\bsni Recent progress by Darvas gives important results concerning the metric geometry of $\cE^1(L)$. By \cite[Corollary 4.14]{darmabuchigeometry}, the formula
$$d_1(\phi_0,\phi_1)=E(\phi_0)+E(\phi_1)-2E(P(\phi_0,\phi_1))$$
gives a distance on the space of finite-energy metrics, where
$$P(\phi_0,\phi_1)=\sup\{\psi\in\PSH(L),\,\psi\leq \min(\phi_0,\phi_1)\}$$
is called the envelope (or rooftop envelope in \cite{darmabuchigeometry}) of the two metrics. Furthermore, the metric space $(\cE^1(L),d_1)$ is shown to be geodesic (\cite[Theorem 4.17]{darmabuchigeometry}), and geodesics can be approximated by geodesics between Fubini-Study metrics (\cite[Theorem 4.3]{darquantization}). Such geodesics arise as the Perron envelope of the endpoints (also called the maximal psh segment joining the endpoints), i.e. as the supremum of all psh segments
$$(0,1)\ni t\mapsto \phi_t$$
such that $\lim_{t\to i}\phi_t \leq \phi_i$, $i=0,1$. Here, we define a psh segment on $L$ (also called \textit{psh paths} or \textit{subgeodesics} in the literature) to be a $\bbs^1$-invariant psh metric on $p_U^*L$, with $p_U:X\times U\to X$, where $U$ is a connected $\bbs^1$-invariant subset of $\bbc^*$. We recover the "time variable" $t$ by setting $t=-\log|z|$ and identifying $U$ with the (real) interval $\{-\log|z|,\,z\in U\}$. Due to Demailly's regularization Theorem, we can characterize a psh segment as a decreasing limit of segments with $\phi_t\in\cH(L)$ for all $t$.

\bsni The purpose of this article is to develop an analogous theory in the ample, non-Archimedean setting. We mention that finite-energy geodesics, and more generally the metric structure of $\cE^1(L)$, have been central to many important recent developments in complex geometry, especially regarding the Yau-Tian-Donaldson conjecture, as in e.g. \cite{bfj}, \cite{ltw} where finite-energy geodesic rays are used to study uniform Ding-stability. 

\paragraph{Non-Archimedean pluripotential theory.} We now explain how to define those objects in the non-Archimedean picture. Consider a polarized projective $\field$-variety $(X,L)$, with $\field$ a non-Archimedean valued field complete with respect to the topology induced by its absolute value $|\cdot|$, and $L$ an ample line bundle. We wish to study the equivalent of $\cH(L)$ and the space of $L$-plurisubharmonic metrics on the Berkovich analytification of $(X,L)$. In the complex case, plurisubharmonic metrics on a bundle $L$ are characterized as the smallest class of metrics stable under finite maxima, addition of constants, decreasing limits, and containing all metrics of the form $k\mi\log|s|$, with $s$ a section of $kL$.

\bsni One then defines the set of non-Archimedean Fubini-Study metrics as
$$\mathcal{H}(L)=k\mi \log\max_i |s_i| + \lambda_i,$$
where $(s_i)$ is a finite basepoint-free collection of sections of $H^0(kL)$ and the $\lambda_i$ are constants. This will be the non-Archimedean counterpart of the complex Fubini-Study metrics, a subclass of $\cH(L)$, which we will also denote by $\cH(L)$. By allowing decreasing limits, one obtains the class of plurisubharmonic functions $\PSH(L)$. It is shown in \cite{bjsemi} that $\PSH(L)$ then satisfies the same stability properties as in the complex case.

\paragraph{Finite-energy metrics.} Given $d=\dim X$ continuous psh metrics $(\phi_1,\dots,\phi_d)$ on $L$, one can define their (mixed) Monge-Ampère measure $\MA(\phi_1,\dots,\phi_d)$, using e.g. the theory of forms and currents on Berkovich spaces of Chambert-Loir and Ducros (\cite{cld}), or intersection theory. As in the complex setting, we can then define the relative Monge-Ampère energy of two continuous psh metrics 
$$E(\phi_0,\phi_1)=\frac{1}{\vol(L)(d+1)}\sum_{i=0}^d \int_X (\phi_0-\phi_1)\MA(\phi_0^{(i)},\phi_1^{(d-i)}).$$
It is finite on continuous metrics and increasing in the first variable, and therefore admits an extension to all of $\PSH$ by setting
$$E(\phi_0,\refmetric)=\inf\{E(\phi'_0,\refmetric),\,\phi'_0\in C^0(L)\cap\PSH(L),\,\phi'_0\geq\phi_0\},$$
for $\phi_0$ psh and $\refmetric$ any continuous and psh metric; and then
$$E(\phi_0,\phi_1)=E(\phi_0,\refmetric)-E(\phi_1,\refmetric),$$
if at least one of the terms on the right-hand side is finite. This extended energy can take infinite values, and we therefore define the class $\cE^1(L)$ of finite-energy metrics as the set of psh metrics $\phi$ with $E(\phi_0,\refmetric)>-\infty$ for a (hence all) continuous psh metric(s) $\refmetric$. An important consequence of the work of Boucksom-Jonsson is that Monge-Ampère operators extend to the class of finite-energy metrics, although we will not touch on this in the present article.

\bsni In Section \ref{sectiondistancee1}, we define the Darvas distance
$$d_1(\phi_0,\phi_1)=E(\phi_0,P(\min(\phi_0,\phi_1)))+E(\phi_1,P(\min(\phi_0,\phi_1)))$$
on $\cE^1(L)$. Here, 
$$P(\phi_0,\phi_1)=\sup\{ \phi,\,\phi\in\PSH(L),\,\phi\leq\phi_0,\,\phi\leq\phi_1\}.$$
It is not immediate that this rooftop envelope is well-behaved. First, if the metrics are continuous, then $P(\phi_0,\phi_1)$ is only continuous provided we assume continuity of envelopes to hold over $(X,L)$. This is equivalent to a number of important statements in non-Archimedean pluripotential theory, e.g. Lemma \ref{continuityofenvelopes}. While conjectural in the general case, it is expected to always hold if, for example $X$ is normal and $L$ is ample; and it is known in situations that include the most important contexts for complex geometric considerations: when $X$ is smooth, and the base field is discretely or trivially valued and of equal characteristic zero. We review the currently known cases after Lemma \ref{continuityofenvelopes}.

\bsni We define our class of non-Archimedean psh segments. We start with the class of Fubini-Study segments, which is the set of maps of the form
$$t\mapsto k\mi \max_i \left(\log|s_i| + (1-t)\lambda_i + t\lambda'_i\right),$$
for $t\in[0,1]$, a finite basepoint-free collection of sections $(s_i)$ of some $H^0(kL)$, and real constants $\lambda_i$, $\lambda'_i$. Then, allowing decreasing limits, we obtain our class of psh segments, whose properties we study in Section \ref{sect_pshsegments}, and which we again show to be the smallest class of segments stable under finite maxima, addition of constants, and decreasing limits, which contains all Fubini-Study segments. It is then clear what we mean by a maximal psh segment between two metrics $\phi_0$ and $\phi_1$: it bounds by above all psh segments $t\mapsto \psi_t$ with $\psi_0\leq \phi_0$ and $\psi_1\leq \phi_1$.

\bsni The main result of this article is then the following:
\paragraph{Theorem A (\ref{thm_finiteenergyismetric}).} Given $\phi_0,\phi_1\in\cE^1(L)$, we set
$$d_1(\phi_0,\phi_1)=E(\phi_0,P(\phi_0,\phi_1))+E(\phi_1,P(\phi_0,\phi_1))$$.
Then,
\begin{enumerate}
\item $(\cE^1(L),d_1)$ is a metric space;
\item there exists a maximal psh segment $t\mapsto\phi_t$ joining $\phi_0$ and $\phi_1$;
\item $\phi_t\in\cE^1(L)$ for all $t$;
\item the segment $\phi_t$ is a (constant speed) metric geodesic for $d_1$, i.e. there exists a real constant $c\geq 0$ such that
$$d_1(\phi_t,\phi_s)=c\cdot |t-s|$$
for all $t,s\in[0,1]$;
\item the Monge-Ampère energy is affine along $\phi_t$, and it is the unique psh segment joining $\phi_0$ and $\phi_1$ with this property.
\end{enumerate}
Furthermore, point (1) holds even without assuming continuity of envelopes.

\paragraph{Strategy of proof, and the continuous case.} There are two keystones to the proof of Theorem A: the fact that similar statements as those of Theorem A hold in the continuous psh case, which is easier to treat; and the fact that finite-energy metrics can always be approximated by decreasing limits of continuous psh metrics, while all of our objects are stable under decreasing limits.

\bsni Furthermore, our non-Archimedean Fubini-Study metrics are always continuous. Therefore, it seems natural to expect that psh geodesics would be continuous if the endpoints are continuous. We will therefore show the following:
\paragraph{Theorem B (\ref{maxprinciple_continuous}).} Let $\phi_0,\phi_1$ be two continuous psh metrics on $L$. Then, 
\begin{enumerate}
\item there exists a unique maximal psh segment $(t,x)\mapsto \phi_t(x)$ joining $\phi_0$ and $\phi_1$;
\item this segment is continuous in both variables;
\item this segment is a geodesic segment for the distance $d_1$;
\item the Monge-Ampère energy is affine along this segment, and it is the unique psh segment joining $\phi_0$ and $\phi_1$ with this property.
\end{enumerate}

\noindent The first step is to show that the supremum of all psh segments  between continuous psh metric is again a psh segment, inspired by variational techniques from e.g. \cite{rwnanalytic}, using Legendre duality. Then, in order to establish the last two points of Theorem B, we approximate (or "quantize") maximal psh segments via psh segments coming from norms acting on spaces of plurisections of $L$. We briefly describe the construction.

\bsni A natural way to understand Fubini-Study metrics is by considering them them as images of non-Archimedean norms $\norm$ acting on $H^0(kL)$ via the Fubini-Study operators
$$\FS_k:\norm \mapsto k\mi\log\sup_{s\in H^0(kL)} \frac{|s|}{\norm(s)},$$
which, when the norm admits an orthonormal basis $(s_i)$ in the non-Archimedean sense, i.e. such that for any $s=\sum a_i s_i\in H^0(kL)$, 
$$\norm(s)=\max_i |a_i|\cdot\norm(s_i),$$
reduces to the expression
$$\FS_k(\norm)=k\mi\log\max_i\frac{|s_i|}{\norm(s_i)}.$$
In particular, Fubini-Study segments can arise as the image by some Fubini-Study operator of metric geodesics in the space of norms over $H^0(kL)$. It is well known that such geodesics exist, from the theory of Bruhat-Tits buildings and their completions, which are CAT(0) spaces. In the case when both norms  are diagonalized, they can be expressed explicitly as the unique curve of norms whose values on the diagonalizing basis are the log-convex interpolations of the values of the endpoint norms. They are geodesic for a multitude of distances inherited from the metric geometry of real affine space (in particular, for the Euclidean distance), hence the denomination. 

\bsni Inspired by \cite{phongsturm}, \cite{berndtprob}, or \cite{darquantization}, we then find a different characterization of maximal psh segments as limits as $k\to\infty$ of certain "finite-dimensional" or "quantum" Fubini-Study segments. For all $k$, we can associate to a continuous psh metric $\phi$, a sup-norm acting on $H^0(kL)$:
$$\N_k(\phi)=\sup_X |s|_{k\phi},$$
and given two such metrics $\phi_0$, $\phi_1$, we set $\phi_t^k$ to be the Fubini-Study segment obtained as the image of a segment $\norm^t_k$ joining $\N_k(\phi_0)$ and  $\N_k(\phi_1)$ in the space of norms on $H^0(kL)$. We then consider the limit
$$\phi_t=\lim \phi_t^k.$$
That this limit exists follows from Theorem \ref{submultgeodesics}, and we show in Section \ref{sect_quantization} that we recover our maximal psh segment as this limit. Since geodesics of norms behave quite nicely, it is then easy to obtain, using quantization results from \cite{boueri} and \cite{reb}, the last two points of Theorem B.

\bsni To prove Theorem A then amounts to showing that the decreasing limit of maximal segments is maximal (for (2)), using the fact that the Monge-Ampère energy is continuous along decreasing nets (for (3) and (5)), and that the $d_1$ distance on $\cE^1(L)$ can also be quantized (for (4), which is proven in Proposition \ref{expressiond1}). Using this quantization, we may also easily prove that $d_1$ satisfies the properties of a distance on $\cE^1(L)$, inherited from the properties of $d_1$ on the space of continuous psh metrics (Proposition \ref{prop_distanceoncontinuous}).

\paragraph{The general case.} Along the way, we also show that any two psh metrics that can be joined by psh segments, without further regularity assumptions, can be joined by a maximal psh segment.

\paragraph{Theorem C (\ref{maxprinciple}).} Let $\phi_0$, $\phi_1$ be any two psh metrics on $L$. Then,
\begin{itemize}
\item either there exists no psh segment between $\phi_0$ and $\phi_1$,
\item or there exists a (unique) maximal psh segment $t\mapsto \phi_t$ between $\phi_0$ and $\phi_1$.
\end{itemize}

\noindent At the time of writing, this is the only construction of non-Archimedean geodesics in the literature, although it has been made known to the author that H. Blum and Y. Liu have independently considered geodesics from our "quantization" point of view, in the trivially valued case, motivated by applications to K-stability, in a work to be published.

\paragraph{Organization of the article.} In Section \ref{sectprerequisites}, we address generalities about non-Archimedean algebra and geometry.

\bsni In Section \ref{sectnorms}, we study spaces of norms, and geodesics inside them. We then consider a natural class of graded norms, and prove submultiplicativity of geodesics between such graded norms.

\bsni In Section \ref{sectnappt}, we briefly recall some basics of non-Archimedean pluripotential theory: general metrics, Fubini-Study and plurisubharmonic metrics, continuity of envelopes, and define the space of finite-energy metrics. We then study some operators connecting norms and metrics.

\bsni In Section \ref{sect_segments}, we define our class of psh segments, study their properties, and state Theorem C concerning maximal segments.

\bsni In Section \ref{sectnageodesicscontinuous}, we introduce the metric space of continuous non-Archimedean psh metrics and study maximal segments in this space. We establish Theorem B, stating that such segments exist, are geodesic for $d_1$, and satisfy other nice properties. We perform the quantization-inspired approach to obtain properties concerning such segments. Along the way, we prove a non-Archimedean version of the Kiselman minimum principle.

\bsni In Section \ref{sectquantization}, we define the metric structure on $\cE^1(L)$, and prove Theorems A and C.

\section*{Acknowledgements.}

The author would like to thank his advisors Sébastien Boucksom and Catriona Maclean for their continuous support throughout the writing of this article. He also thanks Tamás Darvas for some discussion and remarks on the construction of the $d_1$ metric geometry.

\section*{Notation.}

$\norm$ always denotes a norm on a vector space; $\field$ and $\fieldext$ are fields.

\bsni If $L$ is a line bundle over a variety $X$, $H^0(X,kL)$ or $H^0(kL)$ denotes the space of sections of the $k$-th tensor power of $L$; $h^0(kL)$ is its dimension, and $R(X,L)$ is the section ring $\bigoplus_{k\geq 0} H^0(kL)$.

\bsni The symbol $\wedge$ refers to a minimum (usually of metrics) and $\vee$ a maximum (usually of norms). The symbol $\odot$ refers to the symmetric power.

\newpage

\section{Non-Archimedean geometry.} \label{sectprerequisites}

Let $(\field,|\cdot|)$ be a non-Archimedean valued field, i.e. a field assumed to be Cauchy complete with respect to the topology induced by an absolute value $|\cdot|$ satifying the ultrametric inequality $|x+y|\leq\max\{|x|,|y|\}$ for any two $x$, $y\in\field$.

\bsni We will use two criteria to differentiate between such fields. First, the cardinality of the value group $|\field|$, which, as a subgroup of $\bbr$, is either trivial, discrete, or dense. We will then speak of trivially, discretely, or densely valued fields. Note that any field can be endowed with its trivial valuation $|x|=1$ for all $x\in \field$. 

\bsni The second criterion is the data of the characteristic of $\field$, and of the characteristic of its residue field $\tilde{\field}=\{|x|\leq 1\}/\{|x|<1\}$. We will say that $\field$ is of equal characteristic $p$ if $\mathrm{char}\field=\mathrm{char}\tilde\field=p$, and of mixed characteristic if $\mathrm{char}\field=0$ and $\mathrm{char}\tilde\field=p>1$.

\bsni An additional property we will be interested in is maximal completeness. We say that a non-Archimedean field extension $(\fieldext,|\cdot|')/(\field,|\cdot|)$ is immediate if the value groups $|\fieldext|'$ and $|\field|$ coincide, and if the residue fields $\tilde{\fieldext}$ and $\tilde\field$ also coincide. We say that $\field$ is maximally complete if it admits no proper immediate extension.

\bsni Let $X$ be a projective $\field$-variety, i.e. a geometrically integral, separated, projective scheme over $\field$. Being projective, $X$ is equipped with an ample line bundle $L$.

\bsni There exists an analytification functor $X\mapsto X\an$, which we will call the Berkovich analytification, from the category of $\field$-varieties to the category of $\field$-Berkovich spaces, inducing an equivalence of categories between the categories of coherent sheaves over $X$ and over $X\an$. Briefly, in the affine case where $X$ is the spectrum of some algebra $\mathcal{A}$ of finite type over $\field$, the analytified space $X\an$ is the topological space whose points correspond to multiplicative seminorms on $\mathcal{A}$ which coincide with the absolute value $|\cdot|$ on $\field\subset \mathcal{A}$, endowed with the topology of pointwise convergence.

\bsni The Berkovich analytification satisfies the GAGA principle, and preserves connectedness, compactness and Hausdorff-ness. We direct the reader to \cite{berko} and \cite{poineau} for further details about this construction.

\bsni There exists a distinguished dense subset of $X\an$, the set $X^{\mathrm{div}}$ of divisorial points, corresponding to divisorial valuations on $K(X)$ associated to prime divisors on birational models (in the trivially valued case), or on flat models over the valuation ring of $\field$ (in the non-trivially valued case).

\section{Bounded graded norms, geodesics in spaces of norms.}\label{sectnorms}

\subsection{The space of norms on a non-Archimedean vector space.}

We start by studying norms on $\field$-vector spaces. Fix a $d$-dimensional $\field$-vector space $V$. We will only consider here ultrametric norms, i.e. vector space norms $\norm$ on $V$ satisfying $\norm(v+w)\leq \max\{\norm(v),\norm(w)\}$ for all two $v$, $w\in V$. Denote by $\mathcal{N}(V)$ the space of such norms. Note that it is stable under taking the maximum $\norm\vee\norm'$ of two norms.

\bsni One may construct, from any norm on $V$, a quotient norm associated to a subspace $W\subset V$ by setting
$$\norm_{V/W}([v])=\inf_{w\in W} \norm(v+w)$$
for $[v]\in V/W$; and a tensor product norm on any tensor power $V^{\otimes n}$ of $V$, by setting, for each $v\in V^{\otimes n}$,
$$\norm^{\otimes n}(v) = \inf_{v = \sum_i v^i_1\otimes\dots\otimes v^i_n} \max_i (\norm(v^i_1)\dots\norm(v^i_n)),$$
where the infimum is taken over all decompositions of $v$ as $\sum_i v^i_1\otimes\dots\otimes v^i_n$, with a finite number of indices $i$, and with $v^i_k\in V$ for all $i$, $k$. By combining those constructions, one may then define symmetric and wedge product norms. This will be crucial when we will consider graded norms, as that will allow us to "propagate" a norm on the space of sections of a line bundle to its entire section algebra.

\subsection{Spectral measures.}

We will be interested in orthogonal bases of $V$, i.e. bases $\basislong$ of $V$ diagonalizing norms in the ultrametric sense: for all $v=\sum a_i s_i\in V$, we have
$$\norm(v)=\max_i |a_i|\cdot\norm(s_i).$$

\bsni When $\field$ is maximally complete, we know from \cite[2.4.4]{bgr} (see also \cite[Lemma 1.12]{boueri}) that any norm $\norm$ on the $\field$-vector space $V$ admits an orthogonal basis, and furthermore, for any two norms, there exists a basis diagonalizing both of them (we shall therefore speak of \textit{codiagonalizing bases}). Maximally complete fields include  all trivially or discretely valued fields, but not all densely valued fields are maximally complete (see e.g. \cite{poonen} and \cite{comisha}): some important counterexamples include the completed algebraic closures of $\bbq_p$ (i.e. $\bbc_p$) and $\bbc((t))$. In the literature, orthogonal bases are also often referred to as cartesian bases, see \cite{bgr}.

\bsni The spectrum of two norms $\norm$ and $\norm'$ is then the collection, ordered increasingly and counted with multiplicities, of the
$$\lambda_i(\norm,\norm')=\log \frac{\norm'(s_i)}{\norm(s_i)}$$
with $\basislong$ a basis codiagonalizing both norms. It therefore defines a vector in $\bbr^d=\bbr^{\dim V}$. The spectral measure of those norms is defined as the probability measure on the real numbers
$$\sigma(\norm,\norm')=d\mi\sum \delta_{\lambda_i(\norm,\norm')}.$$
If the norms are not diagonalizable, we can still define these quantities as
$$\lambda_i(\norm,\norm')=\sup_{W\in \bigcup_{k \leq i}\mathrm{Gr}(k,V)}\left[\inf_{w\in W-\{0\}}\log\frac{\norm'(w)}{\norm(w)}\right].$$
They are studied in detail in \cite[2.5]{boueri}.

\bsni Absolute moments of spectral distances define distances on spaces of norms: we set
$$d_p(\norm,\norm')^p=\int_\bbr |\lambda|^p\,d\sigma(\norm,\norm').$$
That such distances do satisfy the triangle inequality is not an easy fact: see \cite[3.1]{boueri}. In this article, we will be primarily interested in the case $p=1$. In this case, there is a simple connection with the volume, i.e. the first moment of the relative spectrum:
$$\vol(\norm,\norm')=\int_\bbr \lambda\,d\sigma(\norm,\norm')=\sum_i \lambda_i(\norm,\norm').$$
If $\norm'\geq \norm$, then the spectrum relative to $\norm$ and $\norm'$ belongs to the positive orthant of $\bbr^d$, hence
$$d_1(\norm,\norm')=\vol(\norm,\norm').$$
In the general case, one sees that
$$d_1(\norm,\norm')=\vol(\norm,\norm\vee\norm')+\vol(\norm',\norm\vee\norm').$$
Volumes are easy to manipulate: they satisfy the cocycle property
$$\vol(\norm,\norm')=\vol(\norm,\norm'')+\vol(\norm'',\norm')$$
and antisymmetry
$$\vol(\norm,\norm')=-\vol(\norm',\norm).$$
We will see later that volumes are deeply connected with bifunctionals arising from pluripotential-theoretic considerations, in the spirit of \cite{bberballs}.

\subsection{Apartments.}

We briefly study the structure of the space of diagonalizable norms on a $\field$-vector space of finite dimension equal to $d$. Pick a basis $\basislong$ of $V$, and the projection
$$\iota_S:\bbr^d\to \cN(V)=\cN^{\mathrm{diag}}(V)$$
defined by sending a vector $\alpha=(\alpha_1,\dots,\alpha_d)$ to the unique norm $\norm$ diagonalized in the basis $S$, and with
$$\norm(s_i)=e^{-\alpha_i}$$
for all $i$. The image of such an injection map is called the apartment $\bba_\basis$ associated to the basis $\basis$. It inherits the geometry of $\bbr^d$, in the sense that for any distance $d_p$, $p\in[1,\infty]$, $\iota_S$ realizes an isometry onto its image for the distances $d_p$ as defined in the previous section.

\bsni The space of diagonalizable norms $\cN^{\mathrm{diag}}(V)$, as the (non-disjoint!) union of all the apartments
$$\bigcup_{\basis\text{ basis of }V}\bba_\basis$$
then inherits a complete Euclidean building structure for $p=2$, in the sense of \cite{bremy}. As two diagonalizable norms may be diagonalized in the same basis, it follows that any pair of norms share some apartment in the building $\cN^{\mathrm{diag}}(V)$. If $\field$ is maximally complete, all norms are diagonalizable, so that $\cN^{\mathrm{diag}}(V)=\cN(V)$. 

\bsni As a Euclidean building, $\cN(V)$ therefore has the nice property that $d_p$-geodesics always exist: given two norms $\norm$, $\norm'$ codiagonalized by a basis $\basis$, a geodesic segment connecting them may be obtained as the image through $\iota_\basis$ of a $d_p$-geodesic segment connecting $\iota_\basis\mi(\norm)$ and $\iota_\basis\mi(\norm')$ in $\bbr^d$. We give an explicit description of such geodesics in what follows. The reader might be interested in consulting the article \cite{girardin}.

\subsection{Norm geodesics.}

For the moment, and until Section \ref{subsectgeneralfield}, we assume all the norms involved to be diagonalizable, and we extend to the general case by density of diagonalizable norms in the space of all (ultrametric) norms.

\bsni Pick a basis $\basislong$ of $V$. In the apartment $\mathbb{A}_\basis$, there is what we call a norm geodesic
$$t\in [0,1] \mapsto \norm_t$$
defined as follows: for all $i$, $\norm_t(s_i)=\norm_0(s_i)^{1-t}\cdot \norm_1(s_i)^t$. From an elementary computation it follows that it is geodesic for all distances $d_p$, as we have said above.

\begin{remark}
In the case $p>1$, the norm geodesic is the only geodesic segment between two norms. However, if $p=1$, there are infinitely many geodesic segments between two norms, reflecting the $d_1$ geometry of $\bbr^d$.
\end{remark}

\noindent A fundamental property of norm geodesics is the following:
\begin{lemma}[Log-convexity of norm geodesics]\label{logconvexityofgeodesics} Let $t\in[0,1]$, let $V$ be a $d$-dimensional $\field$-vector space. Pick two norms $\norm_0$, $\norm_1\in\normspace(V)$, and let $t\mapsto \norm_t$ denote the norm geodesic as defined above. Let $s\in V$. We then have that $\log\norm_t(s)$ is a convex function of $t$:
$$\norm_t(s)\leq \norm_0(s)^{1-t}\norm_1(s)^{t}.$$
\end{lemma}
\begin{proof}
Let $(s_{i})$ be a basis codiagonalizing the endpoints, and write $s$ as $\sum a_i\cdot s_{i}$, so that
\begin{align*}
\norm_t(s)&=\max_i |a_i| \cdot \norm_t(s_{i})\\
&=\max_i |a_i| \cdot \norm_0(s_{i})^{1-t}\norm_1(s_{i})^t\\
&=\max_i |a_i|^{1-t}\norm_0(s_{i})^{1-t}\cdot |a_i|^t\norm_1(s_{i})^t\\
&\leq (\max_i |a_i|\cdot\norm_0(s_{i}))^{1-t}\cdot (\max_i |a_i|\cdot\norm_1(s_{i}))^t=\norm_0(s)^{1-t}\norm_1(s)^t.
\end{align*}
\end{proof}

\noindent We now an important comparison inequality concerning norm geodesics with comparable endpoints. They will be crucial in proving many later results, including the metric convexity of $d_1$ in spaces of norms.

\begin{prop}[Monotonicity of norm geodesics with respect to endpoints]\label{fdmaxprinciple}Let $k\in\mathbb{N}$, and set two couples of norms ($\norm_0$, $\norm_1$) and ($\norm'_0$, $\norm'_1$) acting on $H^0(kL)$. If
$$\norm'_0\leq\norm_0 \text{ and } \norm'_1\leq\norm_1,$$
we then have that for all $t$
$$\norm'_t\leq\norm_t.$$
\end{prop}
\begin{proof}
Assume first all norms to be diagonalizable. Write then a section $s$ of in$H^0(kL)$ as $s=\sum a_i s_i$ where $(s_i)$ is a basis codiagonalizing $\norm_0$ and $\norm_1$ (hence all the $\norm_t$), so that
\begin{align*}
\norm'_t(s)&\leq \max_i |a_i|\cdot \norm'_t(s_i)\\
&\leq \max_i |a_i| \cdot \norm'_0(s_i)^{1-t}\norm'_1(s_i)^t\\
&\leq \max_i |a_i|\cdot  \norm_0(s_i)^{1-t}\norm_1(s_i)^t=\norm_t(s),
\end{align*}
where we have used the ultrametric inequality, log-convexity of $\norm'_t$, the inequalities in the hypotheses, then the definition of a basis diagonalizing $\norm_t$. This concludes the proof. If we do not have maximal completeness, one uses approximations by diagonalizable norms and Proposition \ref{proplimgeodesicsnondiag}(iii) to conclude.
\end{proof}

\subsection{Geodesics and the determinant.}

We study the behaviour of geodesics under taking the determinant of norms, which allows us to prove metric convexity of geodesics. We first recall the following result from \cite{boueri}:
\begin{lemma}\label{determinantproduct}
Let $\norm$ be a norm on a $d$-dimensional $\field$-vector space $V$, and let $\basislong$ be a basis of $V$. If $\basis$ diagonalizes $\norm$, then
$$\det\norm(s_1\wedge\dots\wedge s_d)=\prod_{i=1}^d \norm(s_i).$$
We recall that the determinant of a norm $\norm$ on $V$ is the norm induced by $\norm$ on $\det V=V^{\wedge d}$.
\end{lemma}

\begin{lemma}\label{detofgeodesicisgeodesic}
Let $t\mapsto \norm_t$ be a norm geodesic in $V$. Then, 
$$t\mapsto \det \norm_t$$
is the one-dimensional norm geodesic joining $\det\norm_0$ and $\det\norm_1$ in $\det V$.
\end{lemma}
\begin{proof}By density, we may assume that there exists a basis $\basislong$of $V$ diagonalizing the $\norm_t$ for all $t$. We have:
\begin{align*}
\det \norm_t(s_1\wedge\dots\wedge s_d)&=\prod \norm_t(s_i)\\
&=\prod \norm_0(s_i)^{1-t}\norm_1(s_i)^t\\
&=(\det\norm_0(s_1\wedge\dots\wedge s_d))^{1-t}(\det\norm_1(s_1\wedge\dots\wedge s_d))^t,
\end{align*}
which by definition proves the statement. We have used Lemma \ref{determinantproduct} for the first equality, the diagonalizing property of $(s_i)$ for the second, and Lemma \ref{determinantproduct} again for the third.
\end{proof}

\noindent This yields a proof that the relative volume of geodesics is affine:
\begin{corollary}Given two norm geodesics $t\mapsto \norm_t$, $t\mapsto \norm'_t$ in $V$, the function
$$\vol(\norm_t,\norm'_t)=t\mapsto \log \frac{\det \norm'_t}{\det\norm_t}$$
is affine.
\end{corollary}
\begin{proof}
By density again, without loss of generality we can assume that we can pick a basis $\basislong$ diagonalizing $\norm_t$ for all $t$. We then have that for all $t$,
\begin{align*}
\det\norm_t(s_1\wedge\dots\wedge s_d)&=\prod_{i=1}^d \norm_t(s_i)\\
&=\prod_{i=1}^d \norm_0(s_i)^{1-t}\norm_1(s_i)^t\\
&=(\det\norm_0(s_1\wedge\dots\wedge s_d))^{1-t}(\det\norm_1(s_1\wedge\dots\wedge s_d))^t.
\end{align*}
Furthermore, from Lemma \ref{logconvexityofgeodesics} and Lemma \ref{detofgeodesicisgeodesic} we have that the functions $t\mapsto \det\norm_t$ and $t\mapsto \det\norm'_t$ are log-convex. Combining this with the equality above, we find
$$\log \frac{\det \norm'_t}{\det\norm_t}\leq \log \frac{(\det\norm'_0(s_1\wedge\dots\wedge s_d))^{1-t}(\det\norm'_1(s_1\wedge\dots\wedge s_d))^t}{(\det\norm_0(s_1\wedge\dots\wedge s_d))^{1-t}(\det\norm_1(s_1\wedge\dots\wedge s_d))^t},$$
i.e. the function $t\mapsto \vol(\norm_t,\norm'_t)$ is convex. Note that the argument applies symmetrically to show convexity of $t\mapsto \vol(\norm'_t,\norm_t)=-\vol(\norm_t,\norm'_t)$, i.e. that $t\mapsto \vol(\norm_t,\norm'_t)$ is also concave, proving affineness.
\end{proof}

\noindent We may now establish the following result, building on the proof of \cite[Proposition 5.1]{bdl}:
\begin{corollary}[Metric convexity of norm geodesics]\label{metricconvexityd1}
Given two norm geodesics $t\mapsto \norm_t$, $t\mapsto \norm'_t$ in $V$, we have
$$d_1(\norm_t,\norm'_t)\leq (1-t)d_1(\norm_0,\norm'_0)+td_1(\norm_1,\norm'_1).$$
\end{corollary}
\begin{proof}
If the endpoints are comparable in the same order, i.e. $\norm_0\geq \norm'_0$ and $\norm_1\geq \norm'_1$; or $\norm_0\leq \norm'_0$ and $\norm_1\leq \norm'_1$, this follows immediately from the previous Corollary. In the general case, we have to be careful, as the maximum of norm geodesics is not a priori a norm geodesic. However, from Proposition \ref{fdmaxprinciple}, we have that the geodesic
$$t\mapsto \chi_t$$
joining $\norm_0\vee \norm'_0$ and $\norm_1\vee\norm'_1$ satisfies
$$\chi_t\geq \norm_t,\norm'_t$$
for all $t$, i.e.
$$\chi_t\geq \norm_t\vee\norm'_t.$$
Therefore, $\vol(\chi_t,\norm_t\vee\norm'_t)\leq 0$, and we have
\begin{align*}
d_1(\norm_t,\norm'_t)&=\vol(\norm_t,\norm_t\vee\norm'_t)+\vol(\norm'_t,\norm_t\vee\norm'_t)\\
&=\vol(\norm_t,\chi_t)+\vol(\chi_t,\norm_t\vee\norm'_t)+\vol(\norm'_t,\chi_t)+\vol(\chi_t,\norm_t\vee\norm'_t)\\
&\leq \vol(\norm_t,\chi_t)+\vol(\norm'_t,\chi_t).
\end{align*}
Since $\norm_t,\norm'_t\leq\chi_t$, the statement of the corollary holds for the volumes above, and we have
\begin{align*}
d_1(\norm_t,\norm'_t)&\leq (1-t)\vol(\norm_0,\norm_0\vee\norm'_0)+t\vol(\norm_1,\norm_1\vee\norm'_1)\\
&+(1-t)\vol(\norm'_0,\norm_0\vee\norm'_0)+t\vol(\norm'_1,\norm_1\vee\norm'_1)\\
&=(1-t)d_1(\norm_0,\norm'_0)+td_1(\norm_1,\norm'_1),
\end{align*}
proving the general statement.
\end{proof}

\begin{remark}
From the general theory of metric spaces, a result such as Corollary \ref{metricconvexityd1} ensures existence and good properties of the "cone at infinity" or boundary at infinity of $\cN(V)$, which can be described as equivalence classes of (norm) geodesic rays in $\cN(V)$ staying at bounded distance.
\end{remark}

\subsection{The case of a general field.}\label{subsectgeneralfield}

We now return to the case of an arbitrary non-Archimedean field $\field$. In this case, norms on a $\field$-vector space $V$ are not necessarily diagonalizable; however, the set of diagonalizable norms is dense in the set of all ultrametric norms for the $d_p$-distances, $p\in [1,\infty)$, and furthermore the set of ultrametric norms is complete for those distances, by \cite[Corollary 3.3]{boueri}. Using this result, we will show that norm geodesics still exist between norms in $\cN(V)$, rather than only in $\cN^{\mathrm{diag}}(V)$.

\bsni Consider then two non-necessarily diagonalizable norms $\norm_0$ and $\norm_1\in\normspace(V)$, and approximations $\norm^n_0$, $\norm^n_1$ such that
$$d_1(\norm^n_i,\norm_i)\leq\frac1n$$
for $i=0$, $1$. Define
$$\norm_t=\lim_n \norm_t^n,$$
where $t\mapsto \norm_t^n$ is the norm geodesic joining $\norm_0^n$ and $\norm_1^n$.

\begin{prop}\label{proplimgeodesicsnondiag}
The limit $\norm_t$ above exists, and satisfies the following properties:
\begin{enumerate}
\item it is independent of the approximation;
\item it is controlled uniformly by the bounds: given approximations $(\norm_i^m)_m$, $i=0,1$, of the bounds, we have
$$d_1(\norm_t,\norm_t^m)\leq (1-t)d_1(\norm_0,\norm_0^m)+td_1(\norm_1,\norm_1^m);$$
\item $\norm_t$ is the pointwise limit of the approximating norms: for $s\in V$,
$$\norm_t(s)=\lim_m \norm_t^m(s);$$
\item it is log-convex. 
\end{enumerate}
\end{prop}
\begin{proof}
Pick $m$, $n\geq 0$, and write using metric convexity on the diagonalizable geodesics (Corollary \ref{metricconvexityd1}):
\begin{align*}
d_1(\norm^m_t,\norm^n_t)\leq (1-t)d_1(\norm^m_0,\norm^n_0)+td_1(\norm^m_1,\norm^n_1)
\end{align*}
which establishes that the sequence $(\norm^n_t)$ is $d_1$-Cauchy, thus has a limit in $\normspace(V)$. This also establishes the statement about uniform approximation by passing to the limit in $n$ in the previous inequality. Pick now a second pair of approximations $(\norm_0'^{\,n})$, $(\norm_1'^{\,n})$, and write $\norm_t'^{\,n}$ for the norm geodesic between the adequate bounds for all $n$. Then,
\begin{align*}
d_1(\norm_t^n,\norm_t'^{\,n})&\leq (1-t)d_1(\norm_0^n,\norm_0'^{\,n})+td_1(\norm_1^n,\norm_1'^{\,n})
\end{align*}
which vanishes as $n\to\infty$ since both pairs of approximations have the same limits. For the third point, we recall that all metric structures $d_p$, $p\in[1,\infty]$ are equivalent on spaces of norms (see \cite[3.1]{boueri}). In particular, $d_1$-convergence is equivalent to $d_\infty$-convergence, where
$$d_\infty(\norm,\norm')=\log\sup_{s\in V}\left|\frac{\norm'(s)}{\norm(s)}\right|,$$
which gives pointwise convergence. The fourth point follows from the third by passing to the pointwise limit in the log-convexity inequalities
$$\norm_t^m(s)\leq\norm_0^m(s)^{1-t}\norm_1^m(s)^t.$$
\end{proof}

\subsection{Bounded graded norms.}\label{sect_asymptoticequivalence}

\bsni Let $X$ be a projective $\field$-variety, with $\field$ non-Archimedean, and let $L$ be an ample line bundle on $X$. Consider now the section ring $R(X,L)=\bigoplus H^0(kL)$ of $L$, which is a graded algebra of finite type over $\field$. An algebra norm on $R(X,L)$ compatible with the grading may be characterized as the data of norms $\norm_k$ acting on each $H^0(kL)$, satisfying the following submultiplicativity condition: given $s_k\in H^0(kL)$ and $s_\ell\in H^0(\ell L)$, we must have that
$$\norm_{k+\ell}(s_k\cdot s_\ell)\leq \norm_k(s_k)\cdot\norm_\ell(s_\ell).$$
A sequence of norms on $R(X,L)$ satisfying this condition is called a graded norm. In order to study asymptotic properties of graded norms, using e.g. Fekete's lemma, we need a growth condition that is both natural and ensures that the graded norms does not "blow up". We explain here how to formulate such a condition algebraically.

\bsni Since $L$ is ample, there exists some $k$ such that $R(X,kL)$ is generated in degree one, i.e. the morphisms
$$H^0(kL)^{\odot m}\to H^0(mkL),$$
defined by multiplication are surjective for all $m$. We say that a graded norm $\grnorm$ on $R(X,L)$ is generated in degree one if $R(X,L)$ is generated in degree one, and such that for all $m$, $\norm_{km}$ is the quotient norm induced by the surjective symmetry morphisms as above. Finally, we will say that a graded norm $\grnorm$ on $R(X,L)$ is a bounded graded norm if it has at most exponential distorsion with respect to a norm generated in degree one on the section algebra of some power of $L$, i.e. there exist $k$ such that $R(X,kL)$ is generated in degree one, a graded norm $\grnorm'$ generated in degree one on $R(X,kL)$, and a constant $C>0$ with
$$e^{-kmC}\norm'_m \leq \norm_{km} \leq e^{kmC}\norm'_m$$
for all $m$. We denote by $\mathcal{N}_\bullet(R)$ the set of such graded norms. As we will see next, there are natural metric structures on this space, which will ultimately be intimately related to metric structures on spaces of plurisubharmonic metrics on $L$.

\bsni Given bounded graded norms $\grnorm$, $\grnorm'\in\grnormspace(R)$, the sequence of rescaled measures
$$m\mi_*\sigma(\norm_m,\norm'_m)$$
weakly converges to a compactly supported probability measure (as a consequence of \cite{cmac}, see \cite[4.2]{reb}), which we call the relative spectral measure of $\grnorm$ and $\grnorm'$:
$$\sigma(\grnorm,\grnorm')=\lim_m \,(mh^0(mL))\mi_*\sigma(\norm_m,\norm'_m).$$
Its absolute moments define semidistances on $\grnormspace(R)$, 
$$d_p(\grnorm,\grnorm')^p=\int_\bbr |\lambda|^p \,d\sigma(\grnorm,\grnorm').$$
In fact, the equivalence relation defined by identifying two bounded graded norms at zero $d_p$-distance is independent of $p$; we therefore denote by $\grnormspace(R)/\sim$ the honest metric space obtained in this way.

\bsni It is then easy to see that the asymptotic $d_p$-distances may be recovered as the limit of the finite-dimensional distances
$$m\mi d_p(\norm_m,\norm'_m)^p=m\mi\int_\bbr |\lambda|^p\,d\sigma(\norm_m,\norm'_m).$$
Similarly, the asymptotic volume of two bounded graded norm is defined as
$$\vol(\grnorm,\grnorm')=\lim_m (mh^0(mL))\mi \vol(\norm_m,\norm_m').$$
It retains the properties of finite-dimensional volumes, i.e. the cocycle property and antisymmetry.

\subsection{Geodesics between bounded graded norms.}

We prove now our main result for this Section. Given two bounded graded norms $\grnorm^0$ and $\grnorm^1$, one wonders whether there exists a geodesic of graded norms in $\grnormspace(R)$ for the asymptotic $d_1$-distance (or $d_p$, $p<\infty$). It seems obvious to consider, for all $k$, the norm geodesic $\norm_k^t$ joining $\norm_k^0$ and $\norm_k^1$ in $H^0(kL)$, and to set, for all $t$, 
$$\grnorm^t=(\norm^t_k)_k.$$
There are two points to show here: submultiplicativity, and geodesicity. The latter is rather simple: by definition of our norm geodesics, we have that, for $t$, $t'\in [0,1]$, 
$$d_p(\norm_k^t,\norm_k^{t'})=|t-t'|\cdot d_p(\norm_k^0,\norm_k^1),$$
so that geodesicity follows upon taking the limit in $k$. Showing submultiplicativity is a bit trickier.

\begin{theorem}\label{submultgeodesics}
For all $t\in[0,1]$, the sequence of norms $\grnorm^t$ defined above is submultiplicative.
\end{theorem}

\noindent If the norms are diagonalizable, the proof is rather simple; to pass to the general case, one needs to be careful that our usual pointwise approximation argument is not immediately valid when one considers graded norms. We therefore need some preliminary tools, useful only to study the non-diagonalizable case.

\bsni We first recall the definition of the $d_\infty$-distance (already encountered in the proof of \ref{proplimgeodesicsnondiag}): given two norms $\norm$, $\norm'$ acting on a $d$-dimensional $\field$-vector space $V$,
$$d_\infty(\norm,\norm')=\log \sup_{v\in V-\{0\}} \left|\frac{\norm'(v)}{\norm(v)}\right|.$$
By definition, this distance is characterized as the smallest constant $C\geq 0$ such that
\begin{equation}\label{dinftydistorsion}e^{-C}\norm(v)\leq\norm'(v)\leq e^C\norm(v)\end{equation}
for all $v\in V$, i.e. it is the best constant characterizing the (exponential) distortion between the two norms. As seen in the proof of \ref{proplimgeodesicsnondiag} as well, convergence in $d_\infty$-distance implies pointwise convergence (as a real-valued function on $V$). Finally, one can easily see (at first in the diagonalizable case, then by density) that
$$d_\infty(\norm,\norm')=\max_{i\in\{1,\dots,d\}}|\lambda_i(\norm,\norm')|.$$
It is then possible to prove that an analogue of Corollary \ref{metricconvexityd1} holds:
\begin{lemma}[Convexity of $d_\infty$ along norm geodesics]\label{dinftyconvexity}
Let $\norm_0$, $\norm_1$ and $\norm'_0$, $\norm'_1$ be four norms on $V$, and denote by $\norm_t$, resp. $\norm'_t$ the norm geodesic joining the first two, resp. the last two. Then,
$$d_\infty(\norm_t,\norm'_t)\leq (1-t)d_\infty(\norm_0,\norm'_0)+td_\infty(\norm_1,\norm'_1).$$
\end{lemma}
\begin{proof}
As usual, the general case follows from the diagonalizable case by approximation, therefore we make this assumption.
We have that
\begin{align*}
d_\infty(\norm_t,\norm_t')&=\max_i |\lambda_i(\norm_t,\norm'_t)|\\
&=\max_i \,|(1-t)\lambda_i(\norm_0,\norm'_0)+t\lambda_i(\norm_1,\norm'_1)|\\
&\leq (1-t)\max_i|\lambda_i(\norm_0,\norm'_0)|+t\max_i |\lambda_i(\norm_1,\norm'_1)|,
\end{align*}
and the last term is simply $(1-t)d_\infty(\norm_0,\norm'_0)+td_\infty(\norm_1,\norm'_1)$, which is the desired result.
\end{proof}

\bsni We may now prove the Theorem.
\begin{proof}[Proof of Theorem \ref{submultgeodesics}]
We assume at first that all norms involved are diagonalizable. We start with the following case: let $s_m$ be a section of $H^0(mL)$ which belongs to a basis orthogonal for $\norm^0_m$ and $\norm^1_m$. Define in in the same way $s_n\in H^0(nL)$. We then have that
\begin{align*}
\norm_{m+n}^t(s_m\cdot s_n)&\leq\norm^0_{m+n}(s_m\cdot s_n)^{1-t}\norm^1_{m+n}(s_m\cdot s_n)^{t}\\
&\leq \norm^0_m(s_m)^{1-t}\norm^0_n(s_n)^{1-t}\norm^1_m(s_m)^{t}\norm^1_n(s_n)^{t}\\
&=\norm_{m}^t(s_m)\cdot\norm_{n}^t(s_n),
\end{align*}
where we have used log-convexity (Lemma \ref{logconvexityofgeodesics}) in the first inequality, submultiplicativity of the endpoints in the second inequality, and finally the fact that $s_m$ and $s_n$ belong to bases codiagonalizing the endpoints, so that $\norm_m^t(s_m)=\norm_m^0(s_m)^{1-t}\norm^1_m(s_m)^t$, and the same holds for $n$.

\bsni To pass to the general case, write
$$s_m=\sum a_i s_{m,i},\,s_n=\sum b_j s_{n,j},$$
in their adapted bases, and note that
\begin{align*}
\norm_{m+n}^t(s_m\cdot s_n)&\leq \max_{i,j} |a_i|\cdot |b_j|\cdot \norm_{m+n}^t(s_{m,i}\cdot s_{n,j})\\
&\leq \max_{i,j} |a_i|\cdot |b_j|\cdot \norm_{m}^t(s_{m,i})\cdot\norm_{n}^t(s_{n,j})\\
&\leq (\max_i |a_i|\cdot\norm_m^t(s_{m,i})) \cdot (\max_j |b_j|\cdot\norm_n^t(s_{n,j}))\\
&\leq \norm_m^t(s_m)\cdot \norm_n^t(s_n).
\end{align*}
The second inequality follows from the result we just proved, which applies to the $s_{m,i}$ and the $s_{n,j}$; the fourth inequality follows from the fact that those bases diagonalize $\norm_m^t$ and $\norm_n^t$. This proves the desired result.

\bsni Finally, if the norms are not diagonalizable, pick, for $i=0,1$, $\varepsilon>0$, and all $m\in\mathbb{N}^*$, diagonalizable norms $\norm^{i,\varepsilon}_m$ such that
$$d_\infty(\norm^{i,\varepsilon}_m,\norm^i_m)<\varepsilon.$$
Those always exist by $d_\infty$-density of the set of diagonalizable norms on a finite-dimensional $\field$-vector space. By (\ref{dinftydistorsion}), we then have that, for all $m$:
\begin{equation}\label{submult1}e^{-\varepsilon}\norm^i_m\leq \norm^{i,\varepsilon}_m \leq e^{\varepsilon}\norm^i_m.
\end{equation}
Pick sections $s_m\in H^0(mL)$, $s_\ell\in H^0(\ell L)$. We then have that
\begin{align*}
\norm^{i,\varepsilon}_{m+\ell}(s_m\cdot s_\ell)&\leq e^{\varepsilon}\norm^i_{m+\ell}(s_m\cdot s_\ell)\\
&\leq e^{\varepsilon}\norm^i_m(s_m)\cdot \norm^i_\ell(s_\ell)\\
&\leq e^{3\varepsilon}\norm^{i,\varepsilon}_m(s_m)\cdot \norm^{i,\varepsilon}_\ell(s_\ell).
\end{align*}
We have used the right-hand side of (\ref{submult1}) for the first inequality; submultiplicativity of $\grnorm^i$ for the second inequality, and finally the left-hand side of (\ref{submult1}) for the third one. Multiplying both sides by $e^{3\varepsilon}$ we then have that
$$e^{3\varepsilon}\norm^{i,\varepsilon}_{m+\ell}(s_m\cdot s_\ell)\leq e^{6\varepsilon}\norm^{i,\varepsilon}_m(s_m)\cdot \norm^{i,\varepsilon}_\ell(s_\ell),$$
i.e. the sequence of norms $e^{3\varepsilon}\norm_\bullet^{i,\varepsilon}$ is submultiplicative. As
$$d_\infty(e^{3\varepsilon}\norm^{i,\varepsilon}_m,\norm^i_m)=3\varepsilon+d_\infty(\norm^{i,\varepsilon}_m,\norm^i_m)<4\varepsilon,$$
one can see the $e^{3\varepsilon}\norm^{t,\varepsilon}_m$ to be the norm geodesics joining the $e^{3\varepsilon}\norm^{i,\varepsilon}_m$, and also actually $e^{3\varepsilon}$ times the geodesic joining the $\norm^{i,\varepsilon}_m$). We find
$$d_\infty(e^{3\varepsilon}\norm^{t,\varepsilon}_m,\norm^t_m)\leq (1-t)d_\infty(e^{3\varepsilon}\norm^{0,\varepsilon}_m,\norm^0_m)+td_\infty(e^{3\varepsilon}\norm^{1,\varepsilon}_m,\norm^1_m)<4\varepsilon,$$
thanks to metric convexity of $d_\infty$ (Lemma \ref{dinftyconvexity}). This states that $e^{3\varepsilon}\norm^{t,\varepsilon}_m$
converges pointwise to $\norm^t_m$ for all $m$. As $e^{3\varepsilon}\norm^{i,\varepsilon}_m$ is diagonalizable for all $m$, by the previous case, $e^{3\varepsilon}\norm^{t,\varepsilon}_\bullet$ is submultiplicative, and in particular we have
$$e^{3\varepsilon}\norm^{t,\varepsilon}_{m+\ell}(s_m\cdot s_\ell)\leq e^{6\varepsilon}\norm^{t,\varepsilon}_{m}(s_m)\norm^{t,\varepsilon}_{\ell}(s_\ell)$$
Using the pointwise convergence found above to pass to the limit as $\varepsilon\to 0$, this proves the Theorem.
\end{proof}

\noindent We remark that the last part of the above proof also shows the following:
\begin{prop}
\label{approxsubmult}Let $\grnorm$ be a bounded graded norm on $L$. Then, there exist bounded graded norms $\grnorm^\varepsilon$ on $L$, for all $\varepsilon>0$, satisfying the following properties:
\begin{itemize}
\item $\norm^\varepsilon_m$ is diagonalizable for all $m$;
\item $d_\infty(\norm^\varepsilon_m,\norm_m)<\varepsilon$ for all $m$.
\end{itemize}
\end{prop}

\subsection{Geodesics and the symmetric product.}\label{symproduct}

We conclude this Section with a somewhat isolated result. Previously, we showed that the determinant of a geodesic between two norms is the geodesic between the determinant of said two norms. A more surprising result is that the same also holds for symmetric products, as we will show now. Since symmetric algebras arise as section algebras of line bundles over projective space, this will let us consider easily computable examples in a later section. We assume here for simplicity that $\field$ is maximally complete. Let $V$ be a $d$-dimensional $\field$-vector space.

\bsni We will use the following notation:
\begin{itemize}
\item denote $T^{\odot m}$ the set of increasing sequences of integers between $1$ and $d=\dim V$, and for a basis $\basislong$ and $I\in T^{\odot m}$, $s^{\odot I}=s_{i_1}\odot\dots\odot s_{i_d}$;
\item denote $T^{\otimes m}$ the set of sequences of integers between $1$ and $d=\dim V$, and for a basis $\basislong$ and $J\in T^{\otimes m}$, $s^{\otimes J}=s_{j_1}\otimes\dots\otimes s_{j_d}$;
\end{itemize}

\noindent The following lemma follows from an easy but long computation, which we leave to the discretion of the reader.
\begin{lemma}[Symmetric powers of a norm and diagonalizing bases]\label{sympowernorm}
Let $\norm\in\normspace(V)$ be diagonalized by a basis $\basislong$. Then, for all $m$ positive integers, $\norm^{\odot m}$ is diagonalized by the basis $(s^{\odot I})$ for all $I\in T^{\odot m}$. Furthermore, we have $\norm^{\odot m}(s^{\odot I})=\norm(s_{i_1})\dots\norm(s_{i_d})$. 
\end{lemma}

\noindent We now prove the main result of this subsection:
\begin{lemma}\label{sympowerofgeodesicisgeodesic}Let $t\mapsto \norm_t$ be a norm geodesic in $V$. Let $m$ be a positive integer. Then, 
$$t\mapsto \norm_t^{\odot m}$$
is the norm geodesic joining $\norm_0^{\odot m}$ and $\norm_1^{\odot m}$ in $V^{\odot m}$.
\end{lemma}
\begin{proof}
Pick a basis $\basislong$ of $V$. Due to Lemma \ref{sympowernorm}, it is enough to show that $\norm_t^{\odot m}$ is diagonalized in the basis $(s^{\odot I})_{I\in T^{\odot m}}$, using the notation of Lemma \ref{sympowernorm}; and that given $I\in T^{\odot m}$, we have
$$\norm_t^{\odot m}(s^{\odot I})=\norm_0^{\odot m}(s^{\odot I})^{1-t}\norm_1^{\odot m}(s^{\odot I})^t.$$
This follows immediately from that same Lemma applied to $\norm_t$, $\norm_0$, and $\norm_1$, as in the proof of Lemma \ref{detofgeodesicisgeodesic}.
\end{proof}

\begin{remark}
Note that this result fails for a general quotient of the symmetric algebra. In particular, when looking at the algebra of sections of an ample line bundle $L$, the image of a geodesic in the quotient by the map
$$H^0(kL)^{\odot m}\rightarrow H^0(mkL)$$
may not even be comparable with the geodesic joining the image of the bounds in the general case.
\end{remark}

\section{Non-Archimedean pluripotential theory.}\label{sectnappt}

\subsection{Metrics in non-Archimedean geometry.}

We consider as before $X$ a projective $\field$-variety, with $\field$ non-Archimedean, and $L$ an ample line bundle on $X$. For $x\in X\an$, let $\mathcal{H}(x)$ the completion of the residue field at $x$, endowed with its canonical absolute value.

\bsni A metric $\phi$ on $L$ is the data of functions $$\phi_x:L\otimes \mathcal{H}(x)\to \bbr\cup\{\infty\},$$
for all $x\in X$, such that $|\cdot|_{\phi_x}=e^{-\phi_x}$ is a norm on the one-dimensional $\mathcal{H}(x)$-vector space $L\otimes \mathcal{H}(x)$, and such that for any section $s_U\in H^0(U,L)$ over a Zariski open set $U$, the composition
$$|s_U|_\phi:U\xrightarrow{s_U}L|_U\xrightarrow{|\cdot|_\phi}\bbr_+$$
is continuous. 

\bsni We use the additive conventions from \cite{boueri}, that is:
\begin{itemize}
\item the tensor product of line bundles $L^{\otimes k}\otimes M^{-1}$ is written as $kL-M$; 
\item given metrics $\phi$ on $L$ and $\phi'$ on $M$, the induced metric on $kL-M$ is written as $k\phi-\phi'$; 
\item and we identify a metric $\phi$ with the (possibly singular) function $-\log |1|_\phi$.
\end{itemize}
In particular, we can see the space $C^0(L)$ of continuous metrics on $L$ as an affine space modelled on $C^0(X\an)$, by noticing that $\phi,\,\phi'\in C^0(L)$ transform as $|\cdot|_\phi=|\cdot |_\psi e^{\phi'-\phi}$.

\bsni Fix now a reference metric $\refmetric$ on $L$, and write for the rest of this section
$$|s|=|s|_{\refmetric}$$
for any section $s\in H^0(L)$. All of the objects and properties defined from now on are dependent on $\refmetric$, but we will not make this dependence appear in our notations. In the trivially valued case, one usually chooses $\refmetric$ to be the trivial metric $\phi_0$. 

\subsection{Fubini-Study and plurisubharmonic metrics.}

In the complex setting, the class of plurisubharmonic metrics on a line bundle $L$ is the smallest class of singular metrics containing all metrics of the form $k\mi\log|s|$, for $s$ a nonzero section of $kL$, which is stable under taking finite maxima, decreasing limits, and upon adding constants. Note that this is for the ample case, due to deep results of Demailly. Therefore, we would like a putative class of non-Archimedean plurisubharmonic metrics to satisfy these properties as well. This is the inspiration for the following definitions, following e.g. \cite{bjsemi}.

\begin{defi}
A metric $\phi$ on $L$ is Fubini-Study if there exist a finite basepoint-free collection of sections $(s_i)$ of some $H^0(kL)$, and for each $i$ a constant $\lambda_i\in\bbr$ such that
$$\phi=\max_i(\log|s_i|+\lambda_i).$$
We denote by $\cH(L)$ the set of Fubini-Study metrics on $L$.
\end{defi}

\noindent That the class $\cH(L)$ is stable under finite maxima and addition of constants then follows from the definition. Note that it trivially contains all metrics of the form $k\mi\log|s|$ for nonzero sections $s$ of $kL$. It is then natural to define:
\begin{defi}
A metric $\phi$ on $L$ is ($L$-)plurisubharmonic (or $L$-psh) if it is the limit of a decreasing net of Fubini-Study metrics, and not identically $-\infty$. We denote by $\PSH(L)$ the class of $L$-psh metrics.
\end{defi}

\begin{prop}[{\cite[Proposition 5.6]{bjsemi}}]\label{prop_propertiesofpsh}
The class $\PSH(L)$ satisfies the following properties:
\begin{enumerate}
\item it contains all metrics of the form $k\mi\log|s|$ for nonzero sections $s$ of $kL$;
\item it is stable under:
\begin{enumerate}
\item taking finite maxima;
\item addition of a real constant;
\item limits of decreasing nets;
\end{enumerate}
\item the convex combination of two $L$-psh metrics is a $L$-psh metric;
\item the addition of a $L$-psh metric and a $M$-psh metric yields a $(L+M)$-psh metric;
\item if a net of $L$-psh metrics converges uniformly to a limit metric, then this limit metric is $L$-psh.
\end{enumerate}
\end{prop}
\noindent Note that most of those properties follow directly from the definition, but $2(c)$ is not immediate, as we consider not only decreasing sequences, but decreasing limits of decreasing nets.

\bsni As desired, we then have a similar characterization as in the complex case:
\begin{prop}[{\cite[Corollary 5.10]{bjsemi}}]
The class $\PSH(L)$ is the smallest class of metrics on $L$ satisfying properties $1.$ and $2.$ above.
\end{prop}

\noindent We endow $\PSH(L)$ with the topology of pointwise convergence on the set of divisorial points $X^{\mathrm{div}}\subset X\an$ (as mentioned in Section \ref{sectprerequisites}): a net $\phi_i$ in $\PSH(L)$ is said to converge to $\phi\in\PSH(L)$ if $\phi_i(x)\to\phi(x)$ for all $x\in X^{\mathrm{div}}$. By \cite[Corollary 5.15]{bjsemi}, psh metrics are uniquely determined by their restriction to $X^{\mathrm{div}}$.

\subsection{The Monge-Ampère energy.}\label{mongeampereenergies}

Using either intersection pairings (\cite{gublerinventionestropical}, \cite{boueri}), or the theory of differential forms on Berkovich spaces developed by A. Chambert-Loir and A. Ducros in \cite[5, 6]{cld}, one may define a Radon probability measure associated to $d=\dim X$ continuous plurisubharmonic metrics $\phi_i$ acting on an ample line bundle $L$ over the analytic space $X\an$, denoted
$$\MA(\phi_1,\dots,\phi_d)=V\mi\cdot dd^c\phi_1\wedge\dots\wedge dd^c\phi_d\wedge\delta_X,$$
with $V=(L^d)=\lim_m \frac{h^0(mL)}{m^d\cdot d!}$. For short, if (e.g.) the metric $\phi_1$ appears $n$ times in the expression, we write
$$\MA(\phi_1^{(n)},\dots)=V\mi\cdot (dd^c\phi_1)^n\wedge\dots\wedge\delta_X,$$
and so on; and we set
$$\MA(\phi)=\MA(\phi^{(d)})=V\mi\cdot(dd^c\phi)^d\wedge\delta_X.$$

\bsni We can define a Monge-Ampère energy-type quantity in a relative way (as a bifunctional): given two continuous psh metrics $\phi_0$, $\phi_1$, set
$$E(\phi_0,\phi_1)=\frac{1}{d+1}\sum_{i=0}^d \int_X (\phi_0-\phi_1)\MA(\phi_0^{(i)},\phi_1^{(d-i)}).$$
Note that this is always finite as the metrics are continuous. We recommend \cite[9.3]{boueri} for an in-depth exploration of the properties of the Monge-Ampère energy. The ones of interest to us are the following: given $\phi_0$, $\phi_1$, $\phi_2$ a triple of continuous psh metrics on $L$, we have
\begin{itemize}
\item antisymmetry: $E(\phi_0,\phi_1)=-E(\phi_1,\phi_0)$;
\item a cocycle property: $E(\phi_0,\phi_1)=E(\phi_0,\phi_2)+E(\phi_2,\phi_1)$;
\item that it is increasing in the first argument: if $\phi_0\leq\phi_1$, then $E(\phi_0,\phi_2)\leq E(\phi_1,\phi_2)$.
\end{itemize}
\begin{remark} We would like to briefly address the issue of conventions: we follow those of \cite{bjkstab}, wherein the Monge-Ampère energy is normalized by the volume of $L$. This is not the case in \cite{boueri}.\end{remark}

\noindent The Monge-Ampère energy admits an extension to the class $\PSH(L)$ via
$$E(\phi,\refmetric)=\inf\{E(\psi,\refmetric),\,\psi\geq\phi,\,\psi\in\mathcal{H}\}$$
for a fixed continuous psh metric $\refmetric$. We can also partially extend the relative Monge-Ampère energy, by setting
$$E(\phi,\phi')=E(\phi,\refmetric)-E(\phi',\refmetric)$$
for $\phi$, $\phi'\in \PSH(L)$, and at least one of the two terms in the right-hand side is finite. This extended relative Monge-Ampère energy can therefore take $-\infty$ or $\infty$ as values. This makes $E$ upper semi-continuous for the topology of pointwise convergence on divisorial points, hence continuous along decreasing nets.
\begin{defi}
The class $\cE^1(L)$ of finite-energy plurisubharmonic metrics is defined as the set of $L$-psh metrics $\phi$ satisfying
$$E(\phi,\refmetric)>-\infty.$$
\end{defi}
\noindent Due to the cocycle property of the energy, the class $\cE^1(L)$ is in fact independent of the choice of a reference metric, justifying our choice of notation for this class of metrics of finite energy. In the last section of this article, we will show that this class is in fact a geodesic metric space.

\begin{remark}
Although the details are beyond the scope of this article, a strong motivation to study this class is that mixed Monge-Ampère operators can be extended to $\cE^1$, by the work of Boucksom-Favre-Jonsson (see \cite[Section 6.3]{bfjams}).
\end{remark}

\subsection{Continuity of envelopes.}

Plurisubharmonic metrics are not usually stable under minima. Instead of taking minima, we will consider "rooftop envelopes", as follows. Fix first a metric $\phi$ on $L$, and define
$$P(\phi)=\sup\{\phi'\in\PSH(L),\,\phi'\leq\phi\}.$$
In a similar way, we can define the envelope of a tuple of metrics $\phi_1,\dots,\phi_k$ as
$$P(\phi_1,\dots,\phi_k)=\sup\{\phi'\in\PSH(L),\,\phi'\leq\min_i \phi_i\}.$$
See \cite{darrooftop} for an in-depth study of such envelopes in the complex case. An essential property in pluripotential theory is that for all continuous metrics $\phi$ on $L$, with $X$ normal, $P(\phi)$ is also continuous. If this is true, we say continuity of envelopes holds for $(X,L)$. To give an idea of its importance, we have for example that continuity of envelopes is equivalent to the following result:
\begin{lemma}[{\cite[Lemma 7.30]{boueri}}]\label{continuityofenvelopes} Continuity of envelopes holds for $(X,L)$ if and only if for any family $(\phi_i)_{i\in I}$ of psh metrics on $L$, the upper semicontinuous regularization $usc (\sup_{i\in I}\phi_i)$ is psh.
\end{lemma}

\noindent In analogy with the complex case, continuity of envelopes is conjectured to hold for any polarized variety $(X,L)$ where $X$ is normal (or even unibranch), projective over any complete valued field, and $L$ is an ample line bundle over $X$. It is known if $X$ is smooth, and one of the following holds:
\begin{itemize}
\item when $\field$ is of equal characteristic $0$, and trivially or discretely valued (\cite{bfj});
\item when $\field$ is of equal characteristic $p$, and discretely valued, assuming $X$ to be of geometric origin from a $d$-dimensional family over $\field$, and resolution of singularities on $\field$ in dimension $d+\dim X$ (\cite[Theorem 1.4]{gkjm});
\item when $X$ is a curve, without further hypotheses (\cite{thu}).
\end{itemize}

\subsection{From norms to metrics: the operators $\mathrm{FS}$ and $\mathrm{N}$.}

We now study the connection between norms on spaces of plurisections of $L$, and Fubini-Study metrics. Define the following operators:
\begin{align*}
\FS_k: H^0(kL)\ni \norm &\mapsto k\mi\log\sup_{H^0(kL)-\{0\}}\frac{|s|}{\norm(s)},\\
\N_k: L^\infty(L)\ni \phi&\mapsto \sup_X |s|_{k\phi}.
\end{align*}
We call them respectively the $k$-th Fubini-Study and supnorm operators. They are the non-Archimedean analogues of the complex Fubini-Study and Hilbert operators respectively. If $\basislong$ is a basis of $H^0(kL)$ diagonalizing a norm $\norm$ acting on this space, then
\begin{equation}\label{fubinistudy}\FS_k(\norm)=k\mi\log\max_i \frac{|s_i|}{\norm(s_i)}.
\end{equation}
One can also view the metric $\FS_k(\norm)$ as the quotient metric induced by the surjective evaluation morphism $\cO_X\otimes H^0(kL)\to kL$, as in \cite[Section 3.2]{chenmoriwaki}.

\bsni In particular, the image of any Fubini-Study operator on diagonalizable norms is contained in $\cH(L)$. In fact, our definition of a Fubini-Study metric is much closer to this case than it seems at first:
\begin{lemma}[{\cite[Corollary 7.19]{boueri}}]\label{fubinistudysubspace}
A metric $\phi$ on $L$ is Fubini-Study if and only if there exists for all $m$ divisible enough a basepoint-free basis $\basislong$ of $H^0(mL)$, and a norm $\norm$ on $H^0(mL)$ diagonalized by $\basis$ such that
$$\phi=\max_i (\log|s_i| - \log \norm(s_i)).$$
\end{lemma}

\bsni One can wonder whether $\FS_k\circ \N_k$ and $\N_k\circ \FS_k$ give back the original metric, or the original norm. This is characterized by \cite[Lemma 7.24]{boueri}:
\begin{lemma}\label{comparisonfdenvelope}
Let $\phi$ be a metric on $L$. We then have that:
\begin{enumerate}
\item we have $\FS_m(\N_m(\phi))\leq \phi$;
\item furthermore, equality holds above if and only if $\phi$ is a Fubini-Study metric defined by sections in $H^0(kL)$ such that $k$ divides $m$;
\item we have, for any Fubini-Study metric $\phi=\max_i \log|s_i| + \lambda_i$ defined by a basis of sections $(s_i)$ of a basepoint-free subspace $V$ in $H^0(mL)$, that 
$$\N_m(\phi)\leq\norm,$$ where $\norm$ is any norm on $H^0(mL)$ with $\norm(s_i)=e^{-\lambda_i}$.
\end{enumerate}
\end{lemma}

\noindent However, we can imagine the maps $\phi\mapsto \FS_k\circ \N_k(\phi)$, for all $k$, as "quantizing" the original metric $\phi$; in analogy with the complex case (e.g. the asymptotics of Bergman kernels), we could expect that such maps give back $\phi$ asymptotically. It turns out that, if continuity of envelopes holds for $(X,L)$, this statement is true. We first need a few definitions.

\bsni First, one has to ensure that the limit as $k\to\infty$ of $\FS_k(\norm_k)$ exists and has nice properties, if $\grnorm$ is a bounded graded norm (or an equivalence class of such). Since $\grnorm$ is submultiplicative, we have that the sequence $k\mapsto k\cdot \FS_k(\norm_k)$ is superadditive, while the boundedness condition ensures that this sequence has linear growth, so that by Fekete's lemma, the limit of the $\FS_k(\norm_k)$ exists. By the same lemma, it is in fact a supremum of psh metrics. Thus, assuming continuity of envelopes, the usc regularization of this limit is a psh metric.

\bsni Therefore, assuming continuity of envelopes, we can define
\begin{align*}
\FS_\bullet:(\grnormspace(R)/\sim) \,\ni \grnorm &\mapsto \mathrm{usc}\,\lim_m \FS_m(\norm_m),\\
\N_\bullet:L^{\infty}(L) \ni \phi & \mapsto (\N_m(\phi))_{m\in \bbn-\{0\}}.
\end{align*}
(The equivalence relation is the asymptotic equivalence introduced in Section \ref{sect_asymptoticequivalence}.) We call those operators the asymptotic Fubini-Study and graded supnorm operators. By the previous discussion, and results of \cite{bjkstab} \cite{reb} (i.e. invariance under asymptotic equivalence), the asymptotic Fubini-Study operator is well-defined, while the asymptotic supnorm operator does indeed give a bounded graded norm (\cite[Example 9.2]{boueri}). We then have that
\begin{theorem}[{\cite[Theorem A]{reb}}]\label{theo_reb}
The operator $\FS_\bullet$ is injective modulo asymptotic equivalence, and furthermore $\FS_\bullet\circ \N_\bullet$ is the identity on $\grnormspace(R)/\sim$.
\end{theorem}

\noindent It turns out that the Monge-Ampère energy can also be quantized. In the continuous psh case, which is the one of interest to us for this article, this is the main result of \cite{boueri}. We first introduce the notion of volume of a bounded metric:
\begin{defi}
Let $\phi$, $\phi'$ be bounded psh metrics on $L$. We define their relative volume to be the quantity
$$\vol(\phi,\phi')=\vol(\N_\bullet(\phi),\N_\bullet(\phi')).$$
\end{defi}

\begin{theorem}[{\cite[Theorem 9.15]{boueri}}]\label{theo_be}
Let $\phi$, $\phi'$ be continuous psh metrics on $L$. We have
$$E(\phi,\phi')=\vol(\phi,\phi').$$
\end{theorem}

\noindent It is important to emphasize the similarity with the complex case again: the statement of Theorem \ref{theo_be} is the exact analogue of \cite[Theorem A]{bberballs}.

\section{Plurisubharmonic segments in non-Archimedean geometry.}\label{sect_segments}

We now introduce a class of segments between psh metrics. Again, we would like our class to mimick the properties of psh segments (also usually called subgeodesics) in the complex case: as such segments can be seen as a particular class of psh metrics, we would like our non-Archimedean segments to be the minimal class of convex paths stable under finite maxima, decreasing limits, addition of a constant, and containing a certain class of "simple" segments.

\subsection{Fubini-Study segments, plurisubharmonic segments.}\label{sect_pshsegments}

The basic building block for our plurisubharmonic segments are what we call Fubini-Study segments, which we define as follows.
\begin{defi}
A Fubini-Study segment is a map
$$[0,1]\ni t\mapsto \phi_t\in\cH(L)$$
such that there exist a finite basepoint-free collection of sections $(s_i)$ of some $H^0(kL)$, and for each $i$, real constants $\lambda_i$ and $\lambda'_i\in\bbr$ such that for all $t$,
$$\phi_t=k\mi\max_i \log|s_i| + (1-t)\lambda_i + t\lambda'_i.$$
\end{defi}
\noindent Note the similarity with our definition of Fubini-Study metrics. Again, such segments are immediately seen to be convex in $t$, stable under finite maxima and addition of constants. 
\begin{remark}
In particular, the image by the operator $\FS_k$ of some norm geodesic $t\mapsto\norm_t$ in $H^0(X,kL)$, with $\norm_{0,1}$ diagonalizable, defines a Fubini-Study segment: indeed, given a basis $\basislong$ codiagonalizing the endpoints, we have for all $t$, $i$
$$\norm_t(s_i)=\norm_0(s_i)^{1-t}\norm_1(s_i)^t$$
so that
$$FS_k(\norm_t)=\max_i\left( \log|s_i| - (1-t)\log\norm_0(s_i) - t\log\norm_1(s_i)\right).$$
\end{remark}
\bsni Then, following the idea that psh metrics are decreasing limits of Fubini-Study metrics, we define
\begin{defi}
A plurisubharmonic segment or psh segment is a map
$[0,1]\to\PSH(L)$
which is a decreasing limit of a net Fubini-Study segments.
\end{defi}

\begin{prop}
The class of psh segments is the smallest class of segments
$$[0,1]\to\PSH(L)$$
which contains all segments of the form
$$t\mapsto k\mi(\log|s|+(1-t)\lambda + t\lambda'),$$
for $s$ a section of some $kL$ and $\lambda,\lambda'\in\bbr$, is stable under finite maxima, addition of constants, and decreasing limits of nets.
\end{prop}
\begin{proof}
If we can show that the set of psh segments on $L$ satisfies all those properties, then it will by definition be the smallest such class. As was the case for the proof of the same result for psh metrics rather than segments, only the property of being stable under decreasing limits is not immediate from the definition. However, using the trick from the proof of \cite[Proposition 5.6(vi)]{bjsemi}, we can reduce to the case of a decreasing net of Fubini-Study paths, and by stability under maximum we may also assume that our net $(\phi_{t,\alpha})_\alpha$ only contains segments of the form
$$t\mapsto \phi_{t,\alpha}= k_\alpha\mi(\log|s_\alpha|+(1-t)\lambda_\alpha + t\lambda_\alpha').$$
Since our segments are assumed to be decreasing along the net, fixing $t$ gives a decreasing net $(\phi_{t,\alpha})_\alpha$ of $L$-psh metrics, which then converges to a $L$-psh metric $\phi_t$. Therefore, for all $t\in[0,1]$, $\phi_t$ is not identically $-\infty$. The problem is that we do not know whether the nets of constants converge to finite values. But taking $t=0,1$ yields in particular that the nets $k_\alpha\mi(\log|s_\alpha| + \lambda_\alpha)$ and $k_\alpha\mi(\log|s_\alpha| + \lambda'_\alpha)$ decrease to the $L$-psh metrics $\phi_0$ and $\phi_1$. Let $x$ be a point on which $\phi_0$ and $\phi_1$ are nonsingular. Then,
$$\gamma=\phi_1(x)-\phi_0(x)=\lim_\alpha (\lambda_\alpha' - \lambda_\alpha)$$
is finite, and constant on the set of all such $x$. Performing this argument for all pairs $a<b\in[0,1]$ shows that $\phi_t$ corresponds to the segment
$$t\mapsto \phi_0 + \gamma \cdot t,$$
which is a psh segment, as desired.
\end{proof}

\noindent Finally, we show that our segments also satisfy the remaining properties of Proposition \ref{prop_propertiesofpsh}.
\begin{prop}
Plurisubharmonic segments satisfy the following properties:
\begin{enumerate}
\item the convex combination of two psh segments is a psh segment;
\item the addition of a $L$-psh segment and a $M$-psh segment is a $L+M$-psh segment;
\item if a net of psh segments converges uniformly to a limit segment, then this limit segment is psh.
\end{enumerate}
\end{prop}
\begin{proof}
We start with (2). The statement follows from the case of Fubini-Study segments. Consider thus two such segments
$$t\mapsto \phi_t= k\mi \max_i \log|s_i| + (1-t)\lambda_i + t\lambda'_i$$
and
$$t\mapsto \psi_t=\ell\mi \max_j \log|t_j| + (1-t)\gamma_j + t\gamma'_j.$$
Then
\begin{align*}
\phi_t+\psi_t&=(k\ell\mi)(\max_i (\log|s_i^\ell| + (1-t)\ell\lambda_i + t\ell\lambda'_i)\\
&+\max_j (\log|t_j^k| + (1-t)k\gamma_j + tk\gamma'_j))\\
&=(k\ell\mi)(\max_{i,j} \log|s_i^\ell t_j^k|+(1-t)(\ell\lambda_i+k\gamma_j)+t(\ell\lambda'_i+k\gamma'_j)),
\end{align*}
which is a Fubini-Study segment on $(k\ell)(M+L)$.

\bsni The third point follows from noticing that we can use sequences rather than nets when dealing with uniform convergence, and then adding constants to reduce to the case of a decreasing limit of psh segments, which converges by definition to a psh segment. Finally, the first point follows again from the Fubini-Study case, from a simple computation similar to the proof of (2).
\end{proof}

\subsection{A maximum principle for Fubini-Study segments.}

\noindent The Fubini-Study operators $\FS_k$ are not injective. Hence, it is pleasant to consider a "minimal" norm in the fibre of a Fubini-Study metric, corresponding to its image by $\N_k$. The following result shows that Fubini-Study segments obtained as the image of a norm geodesic joining two such minimal norms is maximal (compare with \cite[Proposition 3.1]{berndtprob}):
\begin{lemma}[Maximum principle for norm geodesics]\label{fdcomparisonprinciple}Set two metrics $\phi_0$, $\phi_1$ in $\cH(L)$ defined by sections in $H^0(kL)$. Let $\tilde\phi_t$ be the Fubini-Study segment obtained as the image by $\FS_k$ of the norm segment joining $\N_k(\phi_0)$ and $\N_k(\phi_1)$. Then, for all $t$, and for all Fubini-Study segments in the image of $\FS_k$ joining $\phi_0$ and $\phi_1$, we have
$$\phi_t\leq\tilde\phi_t.$$
\end{lemma}
\begin{proof}
Note that $\FS_k(\N_k(\phi_{i}))=\phi_{i}$ for $i=0,1$ by Lemma \ref{comparisonfdenvelope}. Now, by Lemma \ref{fubinistudysubspace}, we can write
$$\phi_0=\max_{i}\log|s_i| + \lambda_i$$
and
$$\phi_1=\max_{i}\log|t_i| + \lambda'_j$$
where $(s_i)$ and $(t_i)$ are basepoint-free bases of $H^0(kL)$.
By monotonicity of norm geodesics in the form of Proposition \ref{fdmaxprinciple}, it is enough to show that
$$\N_k(\phi_{0,1})\leq \norm_{0,1},$$
but by Lemma \ref{comparisonfdenvelope}, we have
$$\N_k(\phi_0)=\N_k(\FS_k(\norm_0))\leq \norm_0,$$
and similarly for $\norm_1$, which proves the result.
\end{proof}

\subsection{Maximal psh segments.}

We conclude this section by stating a central Theorem in this article:
\begin{theorem}\label{maxprinciple}
Let $\phi_0$, $\phi_1$ be any two psh metrics on $L$. Then,
\begin{itemize}
\item either there exists no psh segment between $\phi_0$ and $\phi_1$,
\item or there exists a unique maximal psh segment $t\mapsto \phi_t$ between $\phi_0$ and $\phi_1$.
\end{itemize}
\end{theorem}

\noindent We will prove this result in Section \ref{sect_proofthm}. In what follows, we will state and prove versions of this result in larger and larger classes of metrics, starting from the continuous psh case, then finite-energy metrics, and finally general psh metrics. In each case, we show that the maximal segment remains in the same class as the endpoints, for all $t$.

\section{Non-Archimedean geodesics in the space of continuous psh metrics.}\label{sectnageodesicscontinuous}

Throughout this section, we assume that $L$ is an ample line bundle over a projective $\field$-variety $X$, $\field$ non-Archimedean, and that continuity of envelopes holds for $(X,L)$.

\subsection{Main Theorem for continuous psh metrics.}\label{sectiondistancecontinuous}

We start by studying maximal psh segments in the space of continuous psh metrics. This space can be endowed with a metric structure as follows:
\begin{defi}
Consider two metrics $\phi_0$, $\phi_1\in C^0(L)\cap\PSH(L)$. We define
$$d_1(\phi_0,\phi_1)=d_1(\N_\bullet(\phi_0),\N_\bullet(\phi_1)),$$
where the distance in the right-hand side is the distance $d_1$ on bounded graded norms from Section \ref{sectnorms}.
\end{defi}

\begin{remark}
Recall that we defined before
$$\vol(\phi_0,\phi_1)=\vol(\N_\bullet(\phi_0),\N_\bullet(\phi_1)).$$
It follows (see e.g. \cite[Remark 5.4.5]{reb}) that we have the formula
$$d_1(\phi_0,\phi_1)=\vol(\phi_0,P(\phi_0,\phi_1))+\vol(\phi_1,P(\phi_0,\phi_1)).$$
By \cite[Theorem 9.15]{boueri}, this is also equal to
$$d_1(\phi_0,\phi_1)=E(\phi_0,P(\phi_0,\phi_1))+E(\phi_1,P(\phi_0,\phi_1)).$$
This distance is sometimes called the Darvas distance, as it was introduced in \cite{darmabuchigeometry} in the complex case. We will see in Section \ref{sectquantization} that, as in \cite{darmabuchigeometry}, it extends as a distance on the space of finite-energy metrics.
\end{remark}

\begin{prop}
\label{prop_distanceoncontinuous}The $d_1$ distance defined above is indeed a distance on the set of continuous psh metrics.
\end{prop}
\begin{proof}
Symmetry is immediate. The triangle inequality follows from taking the limit in the finite-dimensional triangle inequalities
$$k\mi d_1(\norm_k,\norm_k')\leq k\mi d_1(\norm_k,\norm_k'')+k\mi d_1(\norm_k'',\norm_k')$$
for any three bounded graded norms $\grnorm,\grnorm',\grnorm''\in\grnormspace(R)$. If $d_1(\phi,\phi')=0$, then $N_\bullet(\phi)$ and $\N_\bullet(\phi')$ belong by definition to the same equivalence class of bounded graded norms. Since $\FS_\bullet\circ\N_\bullet$ is the identity on continuous psh metrics and $\FS_\bullet$ factors through asymptotic equivalence, it follows that $\phi=\phi'$. Finally, if $\phi=\phi'$, then their distance is naturally zero.
\end{proof}

\noindent The main Theorem of this section is then the following:
\begin{theorem}\label{maxprinciple_continuous}
Let $\phi_0,\phi_1$ be two continuous psh metrics on $L$. Then, 
\begin{enumerate}
\item there exists a (unique) maximal psh segment $(t,x)\mapsto \phi_t(x)$ joining $\phi_0$ and $\phi_1$;
\item this segment is continuous in both variables;
\item this segment is a geodesic segment for the distance $d_1$;
\item the Monge-Ampère energy is affine along this segment, and it is the unique psh segment joining $\phi_0$ and $\phi_1$ with this property.
\end{enumerate}
\end{theorem}

\subsection{A non-Archimedean Kiselman minimum principle.}

In this section, we prove an auxiliary result, of independent interest, that will help us prove the first two points of Theorem \ref{maxprinciple_continuous}.

\bsni Given a convex function $f:\bbr^p\times\bbr^q\to\bbr\cup\{\infty\}$, it is well-known that the infimum of the marginals
$$\bbr^q\ni y\mapsto \inf_{x\in\bbr^p}f(x,y)$$
is also convex. This generalizes in multiple ways, as Prekopa's theorem (\cite{prekopa}), but also to plurisubharmonic functions (independent of the imaginary part of the variable over which the infimum is taken). This is the well-known Kiselman minimum principle (\cite{kisel1}, \cite{kisel2}), and a crucial tool in the study of plurisubharmonic functions. We propose here a non-Archimedean version of this result.

\begin{lemma}[Non-Archimedean Kiselman minimum principle]\label{kiselman}
Let $[0,1]\ni t\mapsto \phi_t$ be a psh segment in $\PSH(L)$. Then, for each $\tau\in\bbr$, the Legendre transform
$$\hat\phi_\tau:x\mapsto \inf_{t\in[0,1]} \phi_t(x) - t\tau$$
is in $\PSH(L)$. Furthermore, by Legendre duality,
$$\phi_t=\sup_{\tau\in\bbr}\hat\phi_\tau+t\tau.$$
\end{lemma}
\begin{proof}Since a psh segment on $L$ is a global decreasing limit of Fubini-Study segments, and one notices the map between segments of psh metrics
$$(t\mapsto \phi_t)\mapsto(\tau\mapsto\hat\phi_\tau)$$
to be continuous along decreasing sequences of segments (by its definition as an infimum over $t$ of psh metrics), it is enough to consider the case where $t\mapsto\phi_t$ is a Fubini-Study segment, i.e. there exists a finite basepoint-free collection of sections $(s_i)_{i\in I}$ of $H^0(kL)$ for some $k$ such that
$$\phi_t=\max_{i\in I} (\log|s_i| + t\lambda_i + c_i),$$
with fixed constants $(\lambda_i)$ and $(c_i)$.

\bsni Set $\tau\in\bbr$, and consider the functions
$$f:[0,1]\times \bbr^{|I|}\to\bbr,\,(t,s)\mapsto k\mi\max_i s + t(\lambda_i-\tau)+c_i$$
and
$$g:\bbr^{|I|}\to\bbr,\,s\mapsto\inf_{t\in[0,1]}f(t,s).$$
It is clear that $\hat\phi_\tau$ is the composition of $g$ and the formal tropicalization map 
$$\mathrm{trop}:X\ni x\mapsto \left(\log |s_1|(x),\dots,\log \left|s_{|I|}\right|(x)\right)\in ({\bbr\cup\{-\infty\}})^{|I|}.$$
We first show that $g$ is a piecewise-linear convex map, and we will explain how from this result we can prove that $\hat\phi_\tau$ is Fubini-Study.

\bsni The strict epigraph of $g$
$$E_g=\{(s,y)\in \bbr^{|I|}\times \bbr,\,g(s)<y\}$$
is the image under the projection $p:[0,1]\times\bbr^{|I|}\times \bbr \to \bbr^{|I|}\times \bbr$ of the epigraph of $f$
$$E_f=\{(t,s,y)\in [0,1]\times\bbr^{|I|}\times \bbr,\,f(t,s)<y\}.$$
One notices that, since $f$ is piecewise-linear and convex in all variables, $E_f$ is convex and its closure is a piecewise-linear set. Since both of those properties are preserved under linear maps, $E_g$ is also convex and PL. This implies that the same holds for the function $g$. (In particular, putting aside the PL hypothesis, this is precisely the standard proof of the convex infimum principle for marginals.)

\bsni Note that our convex PL function $g$ is increasing in each variable and satisfies, for any real constant $C$,
$$g(s_1+C,\dots,s_{|I|}+C)=g(s_1,\dots,s_{|I|})+k\mi C,$$
because those properties are satisfied by $f$ in the $|I|$ last coordinates, and are preserved upon taking the infimum over the first coordinate. Composing such a function with our formal tropicalization map naturally yields a psh metric (by Remark \ref{rem_fsmetric} below), which proves our result.
\end{proof}

\begin{remark}\label{rem_fsmetric}
We have claimed that, given a basepoint-free basis of sections $\basislong{}_{\in I}$ of $H^0(kL)$ and a convex PL function $f$ of $p=|I|$ real variables, increasing in each variable, and satisfying
$$f(z_1+C,\dots,z_p+C)=f(z_1,\dots,z_p)+k\mi C,$$
for all real constants $C$, then $f(\log|s_1|,\dots,\log|s_p|)$ is a Fubini-Study metric. We will show this with $k=1$ for clarity, and all the arguments below can be adapted for general $k$ upon dividing where needed. Indeed, since $f$ is convex, PL, and satisfies the property above, there exist finitely many affine functions $f_j$ such that
$$f=\max_j f_j$$
and
$$f_j(z_1,\dots,z_p)=\sum_i \alpha_{i,j}z_i + b$$
with $\sum_i a_{i,j}=1$. Since a maximum of psh metrics is psh, it is enough to prove that $f_j(\log|s_1|,\dots,\log|s_p|)$ is psh. We will therefore drop the subscript $j$ and write $a_i$ for the coefficients above.

\bsni Now, the monotonicity condition ensures that the $a_i$ all belong to $[0,1]$, i.e. the vector $(a_{i})_i$ is in the $p$-dimensional simplex. We assume at first that $a_i=p_i/q_i\in\bbq\cap[0,1]$. Denote
$$\alpha_i=p_i \cdot q_i\mi \cdot \prod_j q_j$$
and remark that, by the simplex condition,
$$\sum_i \alpha_i = \prod_i q_i.$$
Then,
\begin{align*}
f(\log|s_1|,\dots,\log|s_p|)&=b+\sum_i a_i\log|s_i|\\
&=\frac{b\cdot \prod_i q_i}{\prod_i q_i}+\frac1{\prod_i q_i}\log\prod_i |s_i^{\alpha_i}|.
\end{align*}
Now, $\prod_i s_i^{\alpha_i}$ is a section of $(\sum_i \alpha_i)L=(\prod_i q_i)L$. Therefore, $f$ is in the image of $FS_{\prod_i q_i}$, i.e. it is a Fubini-Study metric. If some of the coefficients are irrational, then $f$ can be uniformly approximated by a function with rational coefficients satisfying all the conditions above, i.e. $f(\log|s_1|,\dots,\log|s_p|)$ can be uniformly approximated by Fubini-Study metrics, which shows that it is psh.
\end{remark}

\begin{remark}\label{remarkkiselman}
We have stated our minimum principle so as to match the form it will be used in, in the next subsection. A brief look at the proof shows that it can be generalized to the following statement: given a Fubini-Study "polyhedron" (or psh, upon taking decreasing limits) parameterized as
$$\phi_t:P \times X \ni (t,x)\mapsto \max_i \log|s_i| + \langle \alpha_i,t\rangle + b_i,$$
where $P$ is a convex polyhedral subset of $\bbr^d$ for some $d$, $\alpha_i\in\bbr^d$, $b_i\in\bbr$, and the $s_i$ are sections of some $H^0(kL)$, we have that for all $\alpha\in\bbr^d$,
$$\inf_{t\in P}\phi_t(x) - \langle \tau,t\rangle$$
is psh for all $\tau\in\bbr^d$.
\end{remark}

\subsection{Proof of Theorem \ref{maxprinciple_continuous}, (1) and (2).}

With this principle in hand, we can now prove the first two points of Theorem \ref{maxprinciple_continuous}, following ideas from e.g. \cite{dargppt} and \cite{rwnanalytic} in the complex case. Consider the envelope
$$\hat\phi_\tau=P(\phi_0,\phi_1-\tau)$$
for $\tau\in\bbr$. By continuity of envelopes, this defines a continuous psh metric. The two next lemmas essentially prove our result:
\begin{lemma}\label{theoremcontinuitypsh}
If $\phi_0,\phi_1\in C^0(L)\cap\PSH(L)$, the map $t\mapsto\phi_t$ defined as the Legendre transform
$$\phi_t=\sup_{\tau\in\bbr} (t\tau+\hat\phi_\tau)$$
is a psh segment, which is continuous on $[0,1]\times X\an$.
\end{lemma}
\begin{proof}
We start with continuity. Since $\phi_0$, $\phi_1$ are continuous, by continuity of envelopes $P(\phi_0,\phi_1-\tau)$ is continuous for all $\tau$ as well. Start by choosing a compact interval $$S=[a,b]\subset (0,1).$$ 
\begin{itemize}
\item for large positive $\tau$, and for all $t\in S$, $\phi_1-\tau\leq \phi_0$ (since $\phi_0$, $\phi_1$ are continuous, thus bounded) and $$t\tau+P(\phi_0,\phi_1-\tau)=t\tau+P(\phi_1-\tau)=P(\phi_1)+(t-1)\tau.$$
Since $t-1<0$, $(t-1)\tau$ is very negative while $P(\phi_1)$ is bounded, so that $t\tau+\hat\phi_\tau$ does not contribute to the supremum;
\item for large negative $\tau$, by boundedness again we have $\phi_0\leq \phi_1-\tau$ so that $$t\tau+P(\phi_0,\phi_1-\tau)=t\tau+P(\phi_0)\leq P(\phi_0).$$
\end{itemize}
Therefore, for some constant $C(S)>0$, and for all $t\in S$, 
$$\phi_t=\sup_{\tau\in [-C(S),C(S)]}t\tau + \hat\phi_\tau,$$
a supremum of continuous functions over a compact set, which is therefore continuous. We have proven that $(t,x)\mapsto \phi_t(x)$ is continuous on $(0,1)\times X$.

\bsni We now prove that it is continuous up to the boundary. We start with the case $t=0$. For very small values of $t$, very positive values of $\tau$ will never contribute to the supremum, so that we need only consider values of $\tau$ bounded above by some constant $C_0$. Since we always have
$$P(\phi_0,\phi_1-\tau)\leq P(\phi_0)=\phi_0,$$
it follows that for very small values of $t$, 
$$t\tau+\hat\phi_\tau \leq tC_0 + P(\phi_0,\phi_1-\tau)\leq tC_0+\phi_0.$$
Taking the supremum, we thus have
\begin{equation}\label{eq1legendre}\phi_t-\phi_0\leq tC_0.\end{equation}
By boundedness of $\phi_1$, there exists a negative enough value of $\tau$, say $C_0'$, such that $\hat\phi_{C_0'}=P(\phi_0)=\phi_0$, i.e. for all small enough $t$, there exist some $\tau$ with
$$C_0't + \phi_0 \leq t\tau+\hat\phi_\tau$$
which implies
$$C_0't \leq \phi_t - \phi_0.$$
Combining this with (\ref{eq1legendre}) shows that $\phi_t$ converges uniformly to $\phi_0$ for small enough $t$, which proves continuity at $t=0$. If $t$ is very close to $1$, the argument proceeds in the same way, by noticing that
$$P(\phi_0,\phi_1-\tau)=P(\phi_0+\tau,\phi_1)-\tau.$$

\bsni To show that $t\mapsto\phi_t$ is a psh segment, we consider the net of psh segments
$$I\mapsto \max_{\tau\in I} (t\mapsto t\tau + \hat\phi_\tau),$$
where $I$ belongs to the set of finite collections of elements in $\bbr$, directed by inclusion. This does indeed define a psh segment for all such $I$, as a finite maximum of psh segments. By definition, the limit of this net is $t\mapsto \phi_t$, and it is naturally increasing along inclusion. By Dini's Theorem, this gives a sequence of psh segments converging uniformly to $t\mapsto\phi_t$, which is equivalent to saying that it is a psh segment, proving our result.
\end{proof}

\begin{lemma}\label{perron}
The curve $$t\mapsto\phi_t=\sup_{\tau\in\bbr}t\tau+\hat\phi_\tau$$ is the largest psh segment joining $\phi_0$ and $\phi_1$.
\end{lemma}
\begin{proof}
We first show that $\phi_t$ bounds from above all psh segments between the endpoints. Plurisubharmonic segments are defined as decreasing limits of Fubini-Study segments. Therefore, at the endpoints, Dini's theorem gives uniform convergence, ensuring that if $k\mapsto\psi_t^k$ is a sequence of Fubini-Study segments decreasing to a psh segment $\psi_t$ with $\psi_{0,1}\leq \phi_{0,1}$, then we can assume that for all large enough $k$, $\psi_{0,1}^k\leq\phi_{0,1}$. Therefore, it is enough to treat the case of Fubini-Study segments.

\bsni Consider therefore a Fubini-Study segment $t\mapsto \psi_t$ with $\psi_0\leq\phi_0$, $\psi_1\leq\phi_1$. Let
$$\bbr\ni\tau\mapsto \hat\psi_\tau=\inf_{t\in[0,1]}\psi_t-t\tau$$
be its Legendre transform. By the minimum principle Lemma \ref{kiselman}, $\hat\psi_\tau\in\mathcal{H}(L)$ for all $\tau$. Taking $t=0,1$ we have
$$\hat\psi_\tau\leq \psi_0\leq \phi_0$$
and
$$\hat\psi_\tau\leq \psi_1-\tau\leq\phi_1-\tau.$$
As $\hat\psi_\tau$ is psh, we then have
$$\hat\psi_\tau\leq \hat\phi_\tau=P(\phi_0,\phi_1-\tau).$$
Taking the Legendre transform again, we find
$$\psi_t=\sup_{\tau\in\bbr}\hat\psi_\tau+t\tau\leq\sup_{\tau\in\bbr}\hat\phi_\tau+t\tau= \phi_t,$$
which establishes our first desired result: $\phi_t$ bounds all psh segments by above. By Lemma \ref{theoremcontinuitypsh}, $\phi_t$ is itself a psh segment, which concludes the proof.
\end{proof}

\noindent We then have all the tools in hand to prove the Theorem.
\begin{proof}[Proof of Theorem \ref{maxprinciple_continuous} (1)-(2).]
By Lemma \ref{perron}, the curve $$t\mapsto\phi_t=\sup_{\tau\in\bbr}t\tau+\hat\phi_\tau$$ is equal to
$$\sup\{\phi_t,\,\phi_t\text{ psh segment joining }\phi_0,\phi_1\},$$
and by Lemma \ref{theoremcontinuitypsh}, this segment is a psh segment joining $\phi_0$ and $\phi_1$. Therefore, it is a maximal psh segment, and hence is unique. This establishes (1). The continuity statement (2) also follows from Lemma \ref{theoremcontinuitypsh}.
\end{proof}

\subsection{Quantization with geodesics of bounded graded norms.}\label{sect_quantization}

We now turn to the statements (3) and (4) of Theorem \ref{maxprinciple_continuous}. From our definition of the maximal psh segment as a Perron envelope, it is not obvious how to recover the desired properties. Instead, we will obtain a "quantized" characterization of that segment, using sequences of Fubini-Study segments.

\bsni Let $\phi_0$, $\phi_1$ be two continuous psh metrics as before. To those metrics, we can associate the bounded graded norms $\N_\bullet(\phi_i)$, $i=0,1$. They can be joined by the geodesic of graded norms $\grnorm^t$, where $\norm_k^t$ is the norm geodesic joining the $\N_k(\phi_i)$. By Theorem \ref{submultgeodesics}, for all $t$, $\grnorm^t$ is submultiplicative, so that the limit
$$\Phi_t:t\mapsto \lim_k \FS_k(\norm_k^t)$$
exists for all $t$ by Fekete's lemma. By the same lemma, this is in fact a supremum over $k$. We claim that this limit coincides with the maximal psh segment $\phi_t$ joining $\phi_0$ and $\phi_1$. To that end, we show that $\Phi_t$ bounds all psh segments by above. 

\begin{prop}\label{maxquantized}
Let $\psi_t$ be a plurisubharmonic segment joining two metrics $\psi_0\leq \phi_0$ and $\psi_1\leq \phi_1$ in $C^0(L)\cap\PSH(L)$. We then have that
$$\psi_t\leq\Phi_t$$
for all $t\in[0,1]$.
\end{prop}
\begin{proof}
As in the proof of Lemma \ref{perron}, Dini's theorem gives uniform convergence of a sequence $k\mapsto\psi_t^k$ of Fubini-Study segments decreasing to $\psi_t$, so that for all large enough $k$, $\psi_{0,1}^k\leq\phi_{0,1}$, since we have assumed $\psi_{0,1}\leq \phi_{0,1}$.

\bsni Therefore, it enough to prove the result for all Fubini-Study segments $\psi_t$ with $\psi_{0,1}\leq\phi_{0,1}$. The argument is similar to that of \cite[Proposition 2.12]{darquantization}.

\bsni We start by fixing some notation. As we have just said, we can assume $t\mapsto \psi_t$ to be a Fubini-Study segment in the image of some $\FS_k$, with $\psi_0\leq\phi_0$, $\psi_1\leq \phi_1$. Denote by:
\begin{itemize}
\item $t\mapsto\tilde\psi_t$ the image by $\FS_k$ of the norm geodesic in $H^0(kL)$ joining $\N_k(\psi_0)$ and $\N_k(\psi_1)$;
\item $t\mapsto\Phi_t^k$ the image by $\FS_k$ of the norm geodesic in $H^0(kL)$ joining $\N_k(\phi_0)$ and $\N_k(\phi_1)$.
\end{itemize}
By definition, since $\Phi_t=\sup_m \Phi_t^m$, we have
\begin{equation}\label{maxprincipleone}
\Phi_t^k\leq\Phi_t.
\end{equation}
Since $\psi_0=\tilde\psi_0$, $\psi_1=\tilde\psi_1$, by the maximum principle for norm geodesics Lemma \ref{fdcomparisonprinciple}, we have
\begin{equation}\label{maxprincipletwo}
\psi_t\leq \tilde\psi_t.
\end{equation}
Since the composition of $\FS_k\circ\N_k$ preserves inequalities while $\N_k$ and $\FS_k$ reverse them, and since we have
$$\tilde\psi_0=\FS_k(\N_k(\psi_0))\text{ and }\tilde\psi_1= \FS_k(\N_k(\psi_1))$$
as well as
$$\Phi_0^k=\FS_k(\N_k(\phi_0))\text{ and }\Phi_1^k=\FS_k(\N_k(\phi_1)),$$
we then have
$$
\tilde\psi_0\leq\Phi_0^k\text{ and }\tilde\psi_1\leq\Phi_1^k,
$$
so that, by monotonicity of norm geodesics in the form of Proposition \ref{fdmaxprinciple} (applied to the $\N_k(\psi_{0,1})\geq \N_k(\phi_{0,1})$), we have
\begin{equation}\label{maxprinciplethree}
\tilde\psi_t\leq\Phi_t^k.
\end{equation}
Combining (\ref{maxprincipletwo}), (\ref{maxprinciplethree}), and (\ref{maxprincipleone}), we finally have
$$\psi_t\leq\Phi_t,$$
as desired.
\end{proof}

\begin{theorem}
Let $\phi_0$, $\phi_1$ be two continuous psh metrics. The segment
$$\Phi_t:t\mapsto \lim_k \FS_k(\norm_k^t),$$
where $t\mapsto \norm_k^t$ is the norm geodesic joining $\N_k(\phi_0)$ and $\N_k(\phi_1)$, coincides with the maximal psh segment $\phi_t$ joining $\phi_0$ and $\phi_1$.
\end{theorem}
\begin{proof}
By Proposition \ref{maxquantized}, we have $\Phi_t\geq\phi_t$, since in particular $\phi_t$ is a psh segment joining $\phi_0$ and $\phi_1$. But, by Fekete's lemma, the limit of the sequence $k\mapsto \FS_k(\norm_k^t)$ is in fact a supremum, i.e. $\Phi_t$ is a supremum of a subset of the set of psh segments below $\phi_0$ and $\phi_1$. This gives $\Phi_t\leq\phi_t$, which proves our result.
\end{proof}

\subsection{Proof of Theorem \ref{maxprinciple_continuous}, (3) and (4).}

Using our newfound expression for $\phi_t$, we may now finish the proof of the Theorem.
\begin{proof}[Proof of Theorem \ref{maxprinciple_continuous} (3)-(4).] We start with (4). By the cocycle property, we can set $\refmetric=\phi_1$. By \cite[Theorem 9.15]{boueri}, 
$$E(\phi_t,\phi_1)= \vol(\phi_t,\phi_1).$$
For any $k$, let $\norm_k^t$ be the norm geodesic joining $\N_k(\phi_0)$ and $\N_k(\phi_1)$, and write
\begin{align*}
\vol(\norm^t_k,\norm^1_k)&=h^0(kL)^{-1}\sum \lambda_{i,k}(\norm^t_k,\norm^1_k)\\
&=(1-t)\, h^0(kL)^{-1}\sum \lambda_{i,k}(\norm^0_k,\norm^1_k)\\
&=(1-t)\vol(\norm^0_k,\norm^1_k).
\end{align*}
Taking the limit, we have
$$\vol(\norm^t_\bullet,\norm^1_\bullet)=\vol(\phi_t,\phi_1)=(1-t)\vol(\phi_0,\phi_1).$$
By \cite{boueri} again, this is equal to the energy:
$$E(\phi_t,\phi_1)=\vol(\phi_t,\phi_1).$$
The energy is then affine in the diagonalizable case. If the norms are not diagonalizable, we simply note that the geodesics $\norm^t_k$ can be approximated by diagonalizable geodesics $\norm^t_{k,\varepsilon}$ for which
$$\vol(\norm^t_{k,\varepsilon},\norm^1_{k,\varepsilon})=(1-t)\vol(\norm^0_{k,\varepsilon},\norm^1_{k,\varepsilon})$$
so that at the limit
$$\vol(\norm^t_k,\norm^1_k)=(1-t)\vol(\norm^0_k,\norm^1_k)$$
still holds for all $k$, proving our result.

\bsni We now show that such a psh segment is unique. Fix a reference metric $\refmetric\in C^0(L)\cap\PSH(L)$. Assume $t\mapsto \psi_t$ is another such segment joining two metrics $\phi_0$, $\phi_1\in C^0(L)\cap\PSH(L)$, i.e. it is a psh segment along which the energy is affine. By the maximum principle Theorem \ref{maxprinciple}, we then have
$$\psi_t\leq\phi_t$$
for all $t$, and $t\mapsto E(\phi_t,\refmetric)$, $t\mapsto E(\psi_t,\refmetric)$ are then affine functions with the same endpoints, hence for all $t$
$$E(\phi_t,\refmetric)=E(\psi_t,\refmetric).$$
By Proposition \ref{comparisonenergy}, since $\psi_t\leq\phi_t$ and the energies coincide, we have $\psi_t=\phi_t$.

\bsni Finally, for (3), the same idea as in (4) works: for all $k$, and for all $t$ and $t'$ in $[0,1]$ we find
$$d_1(\norm^t_k,\norm^{t'}_k)=|t-t'|d_1(\norm^0_k,\norm^{1}_k),$$
and we conclude by passing to the limit.
\end{proof}

\begin{remark} However, reflecting the $d_1$-geometry of real Euclidean space, there are many more $d_1$-geodesics than just the unique psh geodesic (e.g., take the reparametrization of the concatenation of the geodesic joining $\phi_0$ and $P(\phi_0,\phi_1)$, and the geodesic joining $P(\phi_0,\phi_1)$ and $\phi_1$). The fact that our segment is maximal at least ensures that it is maximal in the set of $d_1$-geodesics which are also psh segments.
\end{remark}

\subsection{A simple example.}

Using the results from Section \ref{symproduct} and our quantization results, we have a nice example: on projective spaces, Fubini-Study segments are in fact maximal psh segments.
\begin{corollary}
Let $X=\bbp^n$ for some positive integer $n$ and $L=\cO_{\bbp^n}(m)$ for some positive integer $m$. Let $\norm_t$ be a norm geodesic in $H^0(L)$. Then, the psh geodesic $t\mapsto \phi_t$ joining $\FS_1(\norm_0)$ and $\FS_1(\norm_1)$ corresponds to $t\mapsto \FS_1(\norm_t)$.
\end{corollary}
\noindent This does not hold in general: even if the boundary norms are generated in degree $m$ for some $m$, the geodesic is not necessarily in degree $m$ for all $t$, and thus the result only holds "at infinity".
\begin{proof}
Since $R(X,L)=\langle (x^I)_{I \in T^{\odot m}} \rangle_\field^{\odot \bullet}$, with $x=(x_0,\dots,x_n)$, we have that a geodesic $t\mapsto\norm_t$ in $H^0(L)$ induces a geodesic $t\mapsto \norm_t^{\odot k}$ in $H^0(kL)$ for all $k$ through the $k$-th symmetric powers, by Lemma \ref{sympowerofgeodesicisgeodesic}. Therefore, for all $t$, the geodesic is generated in degree one, and we have
$$\FS_k(\norm_t^{\odot k})=\FS_1(\norm_t),$$
proving our result.
\end{proof}

\section{Non-Archimedean geodesics in the space of finite-energy psh metrics.}\label{sectquantization}

\subsection{Main Theorem for finite-energy metrics.}\label{sectiondistancee1}

We now turn to the case of finite-energy metrics. In Section \ref{sectiondistancecontinuous}, we have defined a distance $d_1$ on the space of continuous psh metrics, by setting
$$d_1(\phi_0,\phi_1)=d_1(\N_\bullet(\phi_0),\N_\bullet(\phi_1)),$$
for $\phi_0,\phi_1$ continuous psh, which can be expressed as
$$d_1(\phi_0,\phi_1)=E(\phi_0,P(\phi_0,\phi_1))+E(\phi_1,P(\phi_0,\phi_1)).$$
This is exactly Darvas' formula from \cite{darmabuchigeometry}, which in his paper extends to the class of finite-energy metrics. We would like to use this expression to extend $d_1$ to our non-Archimedean finite energy metrics as well. The first step will be to make sure that the envelope of two finite-energy metrics is a finite-energy metric, which is the content of Proposition \ref{prop_envelopeisfiniteenergy}. Our main Theorem in this section is the following.
\begin{theorem}\label{thm_finiteenergyismetric}
Given $\phi_0,\phi_1\in\cE^1(L)$, we set
$$d_1(\phi_0,\phi_1)=E(\phi_0,P(\phi_0,\phi_1))+E(\phi_1,P(\phi_0,\phi_1))$$.
Then,
\begin{enumerate}
\item $(\cE^1(L),d_1)$ is a metric space;
\item there exists a maximal psh segment $t\mapsto\phi_t$ joining $\phi_0$ and $\phi_1$;
\item $\phi_t\in\cE^1(L)$ for all $t$;
\item the segment $\phi_t$ is a (constant speed) metric geodesic for $d_1$, i.e. there exists a real constant $c\geq 0$ such that
$$d_1(\phi_t,\phi_s)=c\cdot |t-s|$$
for all $t,s\in[0,1]$;
\item the Monge-Ampère energy is affine along $\phi_t$, and it is the unique psh segment joining $\phi_0$ and $\phi_1$ with this property.
\end{enumerate}
\end{theorem}
\noindent The author has been informed that Boucksom-Jonsson proved completeness of $(\cE^1,d_1)$ in the trivially valued case, if and only if continuity of envelope holds, in an article being written; it is expected that this is also true in the general case.

\subsection{Proof of Theorem \ref{thm_finiteenergyismetric} (1).}

As promised, we first show that our distance is well-defined:
\begin{prop}\label{prop_envelopeisfiniteenergy}
Given two metrics $\phi_0$, $\phi_1\in\cE^1(L)$, $P(\phi_0,\phi_1)$ belongs to $\cE^1$.
\end{prop}
\begin{proof}
Fix a continuous $L$-psh reference metric $\refmetric$. Let, for $i=0,1$, $k\mapsto\phi_i^k$ be sequences of continuous psh metrics decreasing to $\phi_i$. Assuming $\refmetric\geq \phi^k_0$ for all (large enough) $k$, we then have from Lemma \ref{lem_auxlemma} and the fact that the distance of two comparable metrics is a volume:
\begin{align*}
0\leq\vol(P(\phi^k_0,\phi^k_1),P(\refmetric,\phi^k_1))&=E(P(\phi^k_0,\phi^k_1),P(\refmetric,\phi^k_1))\\
&\leq E(\phi^k_0,\refmetric),
\end{align*}
Since $E$ and $P$ are continuous along decreasing nets, this gives at the limit
$$0\leq E(P(\phi_0,\phi_1),P(\refmetric,\phi_1))\leq E(\phi_0,\refmetric)<\infty.$$
In particular, using the cocycle property, $E(P(\phi_0,\phi_1),\refmetric)$ is finite for any continuous psh reference metric, hence $P(\phi_0,\phi_1)\in\cE^1(L)$.
\end{proof}

\noindent The following Lemma was used in the proof of the previous Proposition.
\begin{lemma}\label{lem_auxlemma}
Let $\phi_0,\phi_1$ be continuous $L$-psh metrics. Then, for any continuous $L$-psh metric $\phi$, we have
$$d_1(P(\phi_0,\phi),P(\phi_1,\phi))\leq d_1(\phi_0,\phi_1).$$
\end{lemma}
\begin{proof}
This is essentially an asymptotic version of \cite[Lemma 3.1]{bjkstab}. By continuity of envelopes, the two metrics in the left-hand side are continuous (and psh), so that they define bounded graded supnorms via the $\N_\bullet$ operator. By \cite[Theorem 7.27]{boueri},
\begin{align*}
P(\phi_0,\phi)&=\FS_\bullet(\N_\bullet(\phi_0\wedge\phi))\\
&=\FS_\bullet(\N_\bullet(\phi_0)\vee\N_\bullet(\phi))
\end{align*}
(note that the statement of \cite[Theorem 7.27]{boueri} uses the envelope $Q$ which corresponds to the envelope defined by Fubini-Study metrics; but for continuous metrics, $P=Q$ by \cite[Proposition 7.26]{boueri}).
Similarly,
$$P(\phi_1,\phi)=\FS_\bullet(\N_\bullet(\phi_1)\vee \N_\bullet(\phi)),$$
i.e.
$$\N_\bullet(P(\phi_0,\phi))=\N_\bullet(\phi_0)\vee \N_\bullet(\phi)$$
and
$$\N_\bullet(P(\phi_1,\phi))=\N_\bullet(\phi_1)\vee \N_\bullet(\phi).$$
Now, for all $m$, by \cite[Lemma 3.1]{bjkstab},
$$d_1(\N_m(\phi_0)\vee \N_m(\phi),\N_m(\phi_1)\vee \N_m(\phi))\leq d_1(\N_m(\phi_0),\N_m(\phi_1)),$$
which at the limit and using the equalities above yields
$$d_1(\N_\bullet(P(\phi_0,\phi)),\N_\bullet(P(\phi_1,\phi)))\leq d_1(\N_\bullet(\phi_0),\N_\bullet(\phi_1)),$$
i.e. by definition
$$d_1(P(\phi_0,\phi),P(\phi_1,\phi))\leq d_1(\phi_0,\phi_1),$$
as promised.
\end{proof}

\noindent In order to prove that $d_1$ satisfies the triangle inequality, and also to make some later results easier to prove, we will approximate the $d_1$ distance as follows. We approximate two metrics $\phi_0$ and $\phi_1$ in $\cE^1(L)$ by sequences $(\phi^k_0)$, $(\phi_1^k)$ in $C^0(L)\cap\PSH(L)$ (or $\mathcal{H}(L)$). We will show that
$$d_1(\phi_0,\phi_1)=\lim_k d_1(\phi_0^k,\phi_1^k).$$
\begin{prop}\label{expressiond1}
Given two metrics $\phi_0$, $\phi_1\in\cE^1(L)$, and nets $(\phi_0^k)$, $(\phi_1^k)$ in $C^0(L)\cap\PSH(L)$ decreasing to $\phi_0$, $\phi_1$ we have
$$d_1(\phi_0,\phi_1)=\lim_k d_1(\phi_0^k,\phi_1^k).$$
\end{prop}
\begin{proof}
By \cite[Remark 5.4.5]{reb}, i.e. the Darvas formula for $d_1$ on continuous psh metrics, we have for all $k$
\begin{align*}
d_1(\phi^k_0,\phi^k_1)&=E(\phi^k_0,P(\phi^k_0,\phi^k_1))+E(\phi^k_1,P(\phi^k_0,\phi^k_1)).
\end{align*}
$P$ is continuous along monotone (hence decreasing) nets, so that $P(\phi^k_0,\phi^k_1)$ decreases to $P(\phi_0,\phi_1)\in\cE^1(L)$, and the result follows by continuity of the energy along decreasing nets.
\end{proof}

\noindent We may now show that $\cE^1(L)$, endowed with $d_1$, is a metric space.
\begin{proof}[Proof of Theorem \ref{thm_finiteenergyismetric}, (1).]
Symmetry is immediate. The triangle inequality follows from Proposition \ref{expressiond1} and the triangle inequality of $d_1$ on continuous psh metrics, so that we only have to show that our distance does indeed separate points.

\bsni Assume first that $\phi_0\geq\phi_1$, so that the distance is in fact a Monge-Ampère energy. Then, Proposition \ref{comparisonenergy} gives $\phi_0=\phi_1$.

\bsni In the general case, we use Corollary \ref{expressiond1} to find
$$0=d_1(\phi_0,\phi_1)=E(\phi_0,P(\phi_0,\phi_1))+E(\phi_1,P(\phi_0,\phi_1)).$$
Both quantities on the right-hand side are positive, which yields
$$\phi_0=P(\phi_0,\phi_1)=\phi_1,$$
by the previous argument, proving our result.
\end{proof}

\subsection{A result concerning comparable metrics with zero relative energy.}

We have used, in the proof of Theorem \ref{maxprinciple_continuous}, the fact that if two comparable metrics have the same Monge-Ampère energy, then they are equal. The most natural setting for this result is that of finite-energy metrics, and indeed we will use it in its full generality in the proof of Theorem \ref{thm_finiteenergyismetric}. In this section, we prove this result. We will need another bifunctional acting on continuous psh metrics, the $I$ energy.

\begin{defi}
Let $\phi_0$, $\phi_1$ be two continuous $L$-psh metrics. Their relative $I$-energy is defined as
$$I(\phi_0,\phi_1)=\int_X (\phi_0-\phi_1)(\MA(\phi_1)-\MA(\phi_0)).$$
Their relative $J$-energy is defined as
$$J(\phi_0,\phi_1)=-E(\phi_0,\phi_1)+\int_X (\phi_0-\phi_1)\MA(\phi_1).$$
Given a reference metric $\refmetric$, we write
$$I(\phi_0)=I(\phi_0,\refmetric)$$
and
$$J(\phi_0)=J(\phi_0,\refmetric).$$
By \cite[Section 6.3]{bjsemi}, the functionals $I$ and $J$ also admit an extension to $\cE^1(L)$, which is continuous along decreasing nets.
\end{defi}
\noindent Note that we have cy \cite[(3.15)]{bjsemi}
$$\int_X (\phi_0-\phi_1)\MA(\phi_0)\leq E(\phi_0,\phi_1)\leq \int_X (\phi_0-\phi_1)\MA(\phi_0),$$
so that the $I$-energy is always nonnegative. This result relies on a special case of the local Hodge index theorem, as in \cite[Proposition 3.5]{bjsemi}. That the $J$-energy is nonnegative follows from the very expression of the Monge-Ampère energy.

\begin{prop}\label{comparisonenergy}
Let $\phi_0$, $\phi_1\in\cE^1(L).$ If $E(\phi_0,\phi_1)=0$ and $\phi_0\geq\phi_1$, then $\phi_0=\phi_1$.
\end{prop}
\begin{proof}
The main argument has been communicated to the author by S. Boucksom and M. Jonsson, as part of works on finite-energy spaces currently in writing.

\bsni Approximate $\phi_0$ and $\phi_1$ by decreasing nets $\phi_0^k$, $\phi_1^k\in\cH(L)$. Up to taking the maximum of the two sequences, we can assume without loss of generality that for all $k$, $\phi_0^k\geq \phi_1^k$. We have that
$$E(\phi^k_0,\phi^k_1)=\frac1{\dim X+1}\sum_i\int_X (\phi^k_0-\phi^k_1)\MA(\phi^k_0{}^{(i)},\phi^k_1{}^{(\dim X-i)}),$$
in the notations of Section \ref{mongeampereenergies}. Since $\phi_0\geq\phi_1$, all of the terms in the above sum are integrals against positive measures of nonnegative functions, hence they are all positive. In particular,
$$0\leq I(\phi^k_0,\phi^k_1)=\int_X (\phi^k_0-\phi^k_1)\MA(\phi^k_1)\leq (\dim X+1)E(\phi^k_0,\phi^k_1)\to 0,$$
where the vanishing follows from continuity of $E$ along decreasing nets, and the fact that $E(\phi_0,\phi_1)=0$.

\bsni Pick any positive measure $\mu$ that can be expressed as $\MA(\phi)$ for some $\phi\in\cH(L)$, and write for $x\in X\an$
\begin{align*}
&\mu(\{x\})(\phi^k_0(x)-\phi^k_1(x)) - \int_X (\phi^k_0-\phi^k_1)\MA(\phi^k_1)
\\&=\int_X \mu(\{x\}) (\phi^k_0-\phi^k_1)\delta_x-\int_X (\phi^k_0-\phi^k_1)\MA(\phi^k_1)\\
&\leq \int_X (\phi^k_0-\phi^k_1) \,\mu-\int_X (\phi^k_0-\phi^k_1)\MA(\phi^k_1)\\
&\leq \int_X (\phi^k_0-\phi^k_1) (\mu-\MA(\phi^k_1)).
\end{align*}
By \cite[Corollary 3.20]{bjsemi}, given four Fubini-Study metrics $\phi_{i}\in\cH(L)$, $i\in\{0,1,2,3\}$ there exists constants $C,a,b$ depending only on $\dim X$ such that
$$\int_X (\phi_0 - \phi_1)(\MA(\phi_2)-\MA(\phi_3))\leq C\cdot I(\phi_0,\phi_1)^a\cdot I(\phi_2,\phi_3)^a\cdot \max_i J(\phi_i)^b.$$
In our case, we then have
$$\int_X (\phi^k_0-\phi^k_1) (\mu-\MA(\phi^k_1))\leq  C\cdot I(\phi_0^k,\phi_1^k)^a\cdot I(\phi,\phi_1^k)^a\cdot \max(J(\phi_0^k),J(\phi_1^k),J(\phi))^b,$$
recalling that we have defined $\mu=\MA(\phi)$. Now, by continuity of the extensions of $I$ and $J$ along decreasing nets,
$$I(\phi,\phi_1^k)^a\cdot \max(J(\phi_0^k),J(\phi_1^k),J(\phi))^b\to I(\phi,\phi_1)^a\cdot\max(J(\phi_0),J(\phi_1),J(\phi))^b$$
while we have established before that
$$I(\phi_0^k,\phi_1^k)\to 0.$$
We then find that
\begin{align*}
\mu(\{x\})(\phi^k_0(x)-\phi^k_1(x)) - \int_X (\phi^k_0-\phi^k_1)\MA(\phi^k_1)\leq C&\cdot I(\phi_0^k,\phi_1^k)^a\cdot I(\phi,\phi_1^k)^a\\
&\cdot \max(J(\phi_0^k),J(\phi_1^k),J(\phi))^b
\end{align*}
and the right-hand side vanishes, while the left-hand side converges as $k\to\infty$ to the nonnegative quantity
$$\mu(\{x\})(\phi_0(x)-\phi_1(x)).$$
The key point now is to solve the non-Archimedean Monge-Ampère equation in order to find a measure $\mu=\MA(\phi)$ with positive Dirac mass at $x$. Now, we recall that a psh function is uniquely determined by its restriction to the set of divisorial points in $X\an$, and that for any such point $x$ we may find a Monge-Ampère measure $\mu_x$ associated to a projective model of $X$ which has an atom at $x$, as in \cite[Example 8.11]{boueri}. As we then have
$$0\leq \mu_x(\{x\})(\phi_0(x)-\phi_1(x))=0,$$
and $\mu_x(\{x\})>0$, we have that $\phi_0=\phi_1$ on all divisorial points of $X\an$, hence on $X\an$.
\end{proof}

\subsection{Proof of Theorem \ref{maxprinciple}.}\label{sect_proofthm}

Consider now two metrics $\phi_0$, $\phi_1\in\PSH(L)$, and pick decreasing nets $\phi_0^k$, $\phi_1^k$ in $C^0(L)\cap\PSH(L)$ converging to $\phi_0$, $\phi_1$.
\begin{lemma}
Let $\phi_0\leq \phi'_0$, $\phi_1\leq\phi'_1$ be continuous psh metrics, and denote by $\phi_t$, $\phi'_t$ the maximal psh segments joining them. Then, for all $t$, $\phi_t\leq\phi'_t$.
\end{lemma}
\begin{proof}
By definition, if $\phi'_t$ is maximal, it bounds from above all segments joining endpoints bounded above by $\phi'_0$, $\phi'_1$, and the result follows.
\end{proof}
\noindent Therefore, the net $k\mapsto \phi_t^k$ is monotonous, where $\phi_t^k$ is the maximal psh segment joining $\phi_0^k$ and $\phi_1^k$. We claim that 
$$\phi_t :t\mapsto \lim_k \phi_t^k$$
is our desired geodesic segment.

\begin{proof}[Proof of Theorem \ref{maxprinciple}.]If no psh segment exists between $\phi_0$ and $\phi_1$, the first statement of the Theorem is proven. 

\bsni Assume now that there exist psh segments between $\phi_0$ and $\phi_1$. Let $t\mapsto \psi_t$ be such a segment. It is then a decreasing limit of a net of psh segments $t\mapsto \psi_t^k$. For all $k$, let $t\mapsto \psi_t'^k$ denote the maximal psh segment joining $\psi_0^k$ and $\psi_1^k$. By maximality, we have for all $t$, $k$,
$$\psi_t^k\leq \psi_t'^k.$$
In particular, $\lim_k\psi_t^k\leq \lim_k \psi_t'^k$, and both are psh segments between $\phi_0$ and $\phi_1$. This shows that one needs only consider limits of maximal psh segments. Furthermore, by the same argument, one needs only consider sequences with endpoints equal to $\phi_0$ and $\phi_1$.

\bsni Therefore, we must show that given any two nets of maximal segments $\phi_t^k$, $\psi_t^k$, such that the endpoints converge to $\phi_0$ and $\phi_1$, the limits are equal for all $t$:
$$\lim_k \phi_t^k=\lim_k \psi_t^k.$$
But
$$\lim_k \phi_t^k=\lim_n \sup\{\varphi^k_t\text{ psh segment between }\phi_0^k\text{ and }\phi_1^k\},$$
and similarly for $\psi_t^k$. Both nets converge to the limit
$$\sup\{\varphi_t\text{ psh segment between }\phi_0\text{ and }\phi_1\},$$
which depends only of the endpoints $\phi_0$ and $\phi_1$, proving the Theorem.
\end{proof}

\subsection{Proof of Theorem \ref{thm_finiteenergyismetric}, (2) to (5).}

From the previous section, the unique maximal psh segment joining two finite-energy metrics $\phi_0$, $\phi_1$ can be recovered as the limit $\phi_t$ of maximal segments $\phi_t^k$ joining decreasing approximations $\phi_0^k$, $\phi_1^k$ of those metrics, in $C^0(L)\cap\PSH(L)$. It could a priori be the case that this leaves the class $\cE^1(L)$. Theorem \ref{thm_finiteenergyismetric}(3) will ensure that this is not the case.

\begin{proof}[Proof of Theorem \ref{thm_finiteenergyismetric} (2)-(5).]
As we have just discussed, existence (i.e. (2) in the Theorem) is ensured by Theorem \ref{maxprinciple}. For all $k$, and for any reference metric $\refmetric\in C^0(L)\cap\PSH(L)$,
$$t\mapsto E(\phi_t^k,\refmetric)$$
is affine, with coefficient equal to $E(\phi_1^k,\phi_0^k)$. By continuity of the energy along decreasing nets, the limit function
$$t\mapsto E(\phi_t,\refmetric)$$
is therefore affine, with coefficient equal to the (finite) energy $E(\phi_0,\phi_1)$. This gives (3). 

\bsni Furthermore, by Proposition \ref{comparisonenergy}, that $\phi_t$ is the only possible psh segment with the property that the Monge-Ampère energy is affine along it is proven using the same arguments as the proof of Theorem \ref{maxprinciple_continuous}(4), establishing (5). To show that they are geodesics for our extended $d_1$ distance on $\cE^1(L)$ again follows from the fact that the segments $\phi_t^k$ are $d_1$-geodesic, by Theorem \ref{maxprinciple_continuous}: for all $t$, $t'$,
$$d_1(\phi_t^k,\phi_{t'}^k)=|t-t'|d_1(\phi_0^k,\phi_1^k),$$
and taking the limit in $k$, using Proposition \ref{expressiond1}.
\end{proof}

\newpage

\bibliographystyle{alpha}
\bibliography{bib}

\end{document}